%% file: aistats-2024-main.tex
\documentclass[twoside]{article}

\usepackage[accepted]{aistats2024}
\usepackage[round]{natbib}

\usepackage{hyperref}

\usepackage{amsmath}
\usepackage{amssymb}
\usepackage{mathtools}
\usepackage{amsthm}
\usepackage{algorithm}
\usepackage{algorithmic}

\theoremstyle{plain}
\newtheorem{theorem}{Theorem}[section]
\newtheorem{proposition}[theorem]{Proposition}
\newtheorem{lemma}[theorem]{Lemma}
\newtheorem{corollary}[theorem]{Corollary}
\theoremstyle{definition}
\newtheorem{definition}[theorem]{Definition}
\newtheorem{assumption}[theorem]{Assumption}
\theoremstyle{remark}
\newtheorem{remark}[theorem]{Remark}
\renewcommand\cite[1]{\citep{#1}} % change default to citep

\usepackage{enumitem} % for tight list
\input{preamble.sty}

\begin{document}



\twocolumn[

\aistatstitle{Analysis of Kernel Mirror Prox for
Measure Optimization}

\aistatsauthor{ Pavel Dvurechensky
%\\
%dvureche@wias-berlin.de
\And Jia-Jie Zhu
%\\
%zhu@wias-berlin.de
}

\aistatsaddress{\\
\{dvureche, zhu\}@wias-berlin.de \\
Weierstrass Institute for Applied Analysis and Stochastics\\
Berlin, Germany} ]

\begin{abstract}
  By choosing a suitable function space as the dual to the non-negative measure cone,
  we study in a unified framework a class of functional saddle-point optimization problems, which we term the Mixed Functional Nash Equilibrium (MFNE), that underlies several existing machine learning algorithms, such as implicit generative models, distributionally robust optimization (DRO), and Wasserstein barycenters.
  We model the saddle-point optimization dynamics as an interacting Fisher-Rao-RKHS gradient flow when the function space is chosen as a reproducing kernel Hilbert space (RKHS). As a discrete time counterpart, we propose a primal-dual kernel mirror prox (KMP) algorithm, which uses a dual step in the RKHS, and a primal entropic mirror prox step.
  We then provide a unified convergence analysis of KMP in an infinite-dimensional setting for this class of MFNE problems, which establishes a convergence rate of $O(1/N)$ in the deterministic case and $O(1/\sqrt{N})$ in the stochastic case, where $N$ is the iteration counter. As a case study, we apply our analysis to DRO, providing algorithmic guarantees for DRO robustness and convergence.
  %  with probability-metric constraints.
\end{abstract}

    % \newpage 
    \section{INTRODUCTION}
    % Training state-of-the-art large-scale machine learning models typically requires stochastic optimization with non-convex objective functions, which has achieved great empirical success. However, its reliability and computational complexity must also be theoretically analyzed and understood.
    % For example, the potential of deep generative models is remarkable in numerous machine learning applications such as image generation~\cite{goodfellow_generative_2014}, reinforcement learning~\cite{ho_generative_2016}, and molecular dynamics~\cite{noeBoltzmannGeneratorsSampling2019}.
    % However, training generative models effectively has been a challenging topic for machine learning research that attracts both empirical and theoretical research interests.

    Modern analysis of machine learning algorithms in high dimensions leverages the view of optimization over probability measures.
    For example, some state-of-the-art analysis of generative adversarial networks and adversarial robustness formulates the training process as finding a mixed Nash Equilibrium (MNE) in a two-player zero-sum game~\citep{hsieh2019finding,domingo-enrich_mean-field_2020,wang_exponentially_2022,trillos_adversarial_2023,kimSymmetricMeanfieldLangevin2024}.
    The MNE problem seeks the solution, i.e., MNE, of the saddle-point optimization problem
    \begin{equation}
      \begin{aligned}
        \label{eq-mne-intro}
    \inf_{\mu\in \mathcal M}  \sup_{\nu\in \mathcal M} G (\mu, \nu),
      \end{aligned}
    \end{equation}
    where $G(\mu, \nu)$
    is a bi-variate objective function(al) of the probability measures $\mu, \nu$, and $\mathcal M$ is the space of probability measures on a convex compact domain\eg 
    a bounded closed convex subset of $\mathbb R^d$.
    Related to \eqref{eq-mne-intro}, we are also interested in treating (constrained) measure optimization problems of the form
    \begin{equation}
    \begin{aligned}
      \label{eq-meas-opt-intro}
  \inf_{\mu\in B\subset \mathcal M}  {\mathcal{E}} (\mu).
    \end{aligned}
  \end{equation}
  % \jz{This notation suggests the objective/energy changes with states. Why not just use $\mathcal{E}$?} \pdc{The problem was that we used the same notation $\mathcal{E}$ for the objective in function and in measure optimization. My idea was to distinguish them. Maybe then we use $\tilde{\mathcal{E}}$ and $\mathcal{E}$.}

    In machine learning, there exists a commonly-used alternative paradigm to directly optimizing the measures, such as in the MNE problem~\eqref{eq-mne-intro} {or in \eqref{eq-meas-opt-intro}}.
    This paradigm exploits the duality between probability measure space $\mathcal M$ and a dual functional {space,} i.e., instead of optimizing w.r.t. a measure, we solve the {(constrained)} \emph{functional optimization problem}
    \begin{equation}
      % \begin{aligned}
        % \label{eq-general}
      \inf_{f\in \mathcal F} \mathcal V ( f),
      \label{eq-fn-opt}
      % \end{aligned}
    \end{equation}
    where $\mathcal F$ is some set of dual functions.
    Such methods often leverage scalable learning models such as reproducing kernel Hilbert spaces (RKHS) and deep neural networks for parameterizing and manipulating the dual function $f$, instead of directly searching in the space of probability distributions as in \eqref{eq-mne-intro} or \eqref{eq-meas-opt-intro}. 
    We have witnessed success of
    such dual functional approaches in many domains of machine learning research, e.g.,
    generative modeling~\citep{nowozin_f-gan_2016,arjovsky2017wasserstein,gulrajani_improved_2017,korotin_wasserstein-2_2020}
    computing optimal transport~\citep{genevay_stochastic_2016}, (see also \cite{dvurechensky2018computational,jambulapati2019direct,guminov2021combination} for discrete transport setting),
    distributionally robust optimization~\citep{zhu_kernel_2021},
    Wasserstein barycenters~\citep{li_continuous_2020,tiapkin_stochastic_2021,korotin_continuous_2021,dvinskikh2021improved}.
    However, optimization over those functional spaces, e.g., RKHSs or neural networks, is inherently an infinite-dimensional problem that requires careful treatment. In particular, algorithms and their convergence analysis do not directly follow from the finite-dimensional setting.
    % whose convergence analysis does not directly follow from the finite-dimensional setting \pd{and requires }.
    
 In certain situations, the dual problem \eqref{eq-fn-opt} may turn out to be non-smooth and from the algorithmic perspective, it is better to keep both functional and measure variables and consider a primal-dual problem formulation. This drives us to our main problem of interest in this paper. Namely, we formulate the \emph{Mixed Functional Nash Equilibrium} (MFNE),
    whose outer optimization is an optimization problem over a set $\mathcal F$ of functions
    \begin{equation}
      \begin{aligned}
        \label{eq-general}
      \inf_{f\in \mathcal F} \sup_{\mu\in \mathcal M} F ( f, \mu),
      \end{aligned}
    \end{equation}
    which, like the MNE, is a special case of the infinite-dimensional pure NE problem.
    %  the bi-level optimization problem cast in  infinite-dimensional function spaces.
    % In different contexts, such a problem is also referred to as a saddle-point problem, minimax problem, or robust optimization problem.
    \mfneRef has appeared in several cutting-edge algorithms in ML research, which we discuss in detail in Section~\ref{sec-bg-dual-metric}. In these examples, algorithmic treatment is often made on a case-by-case basis and a unified convergence analysis for those applications is still missing. In this paper, we make a first step towards a unified treatment of such applications by considering a subclass of MFNE problems where $\mathcal{F}$ is a subset of RKHS, which we call \emph{Mixed Kernel Nash Equilibrium} (MKNE). A prominent example of such problems, as we show, is the distributionally robust optimization problem with the ambiguity set given by the kernel maximum mean discrepancy (MMD) distance.
    % In such cases of saddle-point optimization problems, a unified convergence analysis for those applications to the primal-dual setting is still missing.
    % To fill this gap, this paper provides the convergence analysis for optimization problems that move in a functional space. 

    % Centered around the \mfneRef,
    We summarize our contributions as the following.
    \begin{enumerate}[noitemsep]
        \item 
        We model the infinite-dimensional continuous-time optimization dynamics of the general functional optimization problem~\eqref{eq-fn-opt} as RKHS gradient flows, which enjoys simple convexity structure.
        Then, motivated by an interacting gradient flow~\eqref{eq-react-eq}-\eqref{eq-kgf-2-main} that couples the Fisher-Rao and RKHS gradient flow,
        we derive a primal-dual kernel mirror prox algorithm for solving \mfneRef with $\mathcal{F}$ being a subset of Hilbert space.
        \item 
        We show that our specialized infinite-dimensional mirror prox algorithm, which we term primal-dual kernel mirror prox (KMP), can be applied to the MKNE problem to obtain a convergence rate $\mathcal O (\frac{1}{N})$ in the deterministic and $\mathcal O (\frac{1}{\sqrt{N}})$ in the stochastic settings, where $N$ is the iteration counter. To the best of our knowledge, it is the first analysis with the kernel mirror prox steps in the dual functional space, which differs from the typical mean-field analysis of measure optimization.
        \item 
        Our framework can be applied to existing learning tasks to derive convergence results.
        As a concrete case study, we apply our analysis to distributionally robust optimization (DRO) to establish a convergence rate via primal-dual kernel mirror prox in the case of MMD ambiguity sets.
        Importantly, our method does not require the loss to be convex w.r.t. the uncertain variable, i.e., $l(\theta; x)$ can be non-convex non-linear in $x$ in \eqref{eq-dro-ipm}, which is much more general than common assumptions in the Wasserstein DRO literature.
        We then provide robustness guarantees for the KMP solution both in the population risk and under distribution shift.
        % To the best of our knowledge, this is the first primal-dual infinite-dimensional continuous optimization algorithm guarantee for metric-ball DRO, especially in the non-linear setting.
        % \item
        % Last but not least, the unification perspective of a few learning tasks, provided by MFNE and MKNE, highlights our Kernel Mirror Prox as a general-purpose functional optimization algorithm for optimizing over probability measures in the dual space, with theoretical convergence guarantee.
        % This is similar to general-purpose algorithms such as Langevin Monte-Carlo, Stein variational gradient descent.
    \end{enumerate}
    
    \section{PRELIMINARIES}
    \paragraph*{Notation.}
    We use $\mathcal M= \probSimplex$ to denote the space of probability measures defined on the closure of a bounded convex domain $\Omega\subset \mathbb R^d$.
    We say that \eqref{eq-general} is \cvxcnc if the inner problem is concave maximization and the outer problem is convex minimization.
    The convexity notion, if not otherwise specified, refers to the regular convexity notion (defined in \eqref{eq-convexity}).
    For PDE gradient flows, the states are functions of both time and space, for example, $u(t, x)$. When there is no ambiguity, we write $u(t):=u(t, \cdot)$ to denote the function at evolutionary time $t$.
    If not otherwise specified, proofs are deferred to the appendix.

    \subsection{Duality of Metrics on Probability Measures}
    \label{sec-bg-dual-metric}
    One reason behind the ubiquity of the functional optimization problem~\eqref{eq-fn-opt} and \mfneRef in machine learning problems is that common probability metrics admit a dual characterization.
    For example, the \emph{optimal transport} distance~\citep{santambrogio_optimal_2015,ambrosio2008gradient}, e.g., $p$-Wasserstein metric $W_p$ with $p\geq 1$, can be characterized in the dual space via the dual Kantorovich problem
    \vspace{-0.3cm}
    \begin{align}
      \label{label-kantorovich-dual}
        {W_p^p} (\mu, \nu)
        =
        &
        \sup_{\psi_1, \psi_2} 
        % &
      \int\psi_1\,\dd\mu+
      \int \psi_2\,\dd\nu
    %   &
      \\
      &
      \textrm{ s.t. }
    %   &
      \psi_1(x) +  \psi_2(y) \leq c(x,y),
    %   &
    \end{align}
    % \vspace{-0.1cm}
    $\forall x, y , \textrm{ a.e.}$
    which is an infinite-dimensional optimization problem with an infinite constraint.
    % $\psi_1(x) +  \psi_2(y) \leq c(x,y),\ \forall x, y , \textrm{ a.e.}$
    Bounded continuous functions $\psi_i$ are referred to as the Kantorovich potential functions. $c(x,y)=\|x-y\|_p^p$ is the transport cost function associated with the transport.
    %,$c(x,y) = W_p$.
    In the machine learning literature, researchers have explored this dual formulation by directly parameterizing the Kantorovich potential $\psi_i$'s, e.g., using RKHS functions \pd{\citep{genevay_stochastic_2016,vacher21dimension-free}} or input convex neural networks (ICNNs) \citep{li_continuous_2020,korotin_wasserstein-2_2020}. 
    
    Another commonly used metric is the integral probability metric (IPM), which is defined via the weak norm formulation
        \begin{equation}
        \label{eq:IPM_def}
           \operatorname{IPM}( \mu, \nu) = \sup_{f\in \mathcal F}\int f \dd (\mu-\nu). 
        \end{equation}
    One particular choice of the test function family is the RKHS-norm-ball
    $\mathcal F=\left\{f : \hnorm{f}\leq 1\right\}$ (here $\mathcal{H}$ is a RKHS), which yields the kernel maximum mean discrepancy (MMD)~\citep{gretton_kernel_2012}.
    The choice
    $\mathcal F=\{f: \operatorname{Lip}{(f)}\leq 1\}$ (i.e., class of $1$-Lipschitz continuous functions) recovers the type-1 (Kantorovich-)Wasserstein metric.
    
    Such basic duality characterizes the relationship of the underlying geometry of the learning problem and the variational problem that optimizes w.r.t. functions, such as the IPM test functions and the Kantorovich potentials in OT.
    This has been exploited in several fields in machine learning.
    % , which we detail below.
    
    \textbf{Implicit generative models (IGM).}
    We consider the following IGM formulation \citep{li_generative_2015,dziugaite_training_2015,goodfellow_generative_2014},
    \begin{align}
      \label{eq-igm}
    \inf_{G_\theta} \mathbb E_Z \mathcal{D}( P ,G_\theta(Z)),
    \end{align}
    where $G_\theta$ is the generator function parameterized by $\theta\in \mathbb{R}^d$, $P$ is a target distribution, $Z$ is a random variable with (simple) latent distribution $Q$, $\mathcal{D}$ can be chosen as a discrepancy measure between distributions such as the
    $f$-divergence family,
    optimal transport distance,
    or kernel MMD.
    In IGM, $P$ is often taken to be the training data distribution.
    Recent theoretical analysis of the optimization for training~\eqref{eq-igm} \emph{lifts} the non-convex optimization problem to the space of probability measures, which is an instantiation of a MNE problem~\eqref{eq-mne-intro}.
    The problem can then be cast into this paper's general MFNE formulation~\eqref{eq-general}, for example, by choosing the function $\mathcal{D}$ from the IPM family \eqref{eq:IPM_def}
    \begin{align}
    \inf_{\mu\in \mathcal M}
     \sup_{f\in \mathcal F}
    \biggl\{
    \int f(x) d P(x) - 
    \mathbb E_{\theta\sim \mu}
    \int f (g_\theta(z)) d Q(z)
    \biggr\}
    ,
    \end{align}
    where $g_\theta$ is the generator.
    If the function class $\mathcal F$ is chosen to be the class of $1$-Lipschitz functions, then the formulation is the Wasserstein GAN~\citep{arjovsky2017wasserstein,gulrajani_improved_2017}.
    On the other hand, if $\mathcal F$ is the RKHS-norm ball, this is the MMD GAN~\cite{li_mmd_2017,binkowski_demystifying_2018}.
    Note that this lifted problem is now \emph{\cvxcnc} in the optimization variables.
    
    \textbf{Distributionally robust optimization (DRO).}
    We now consider the metric-ball-constrained DRO problem~\citep{delage_distributionally_2010,ben2013robust,gao_distributionally_2016,zhao_data-driven_2018}
    %  and its Lagrangian form
    \begin{align}
      \label{eq-dro-ipm}
      \inf_{\theta \in \Theta} \sup_{\mu : \mathcal D(\mu, \emp) \le \epsilon} \E_\mu[l(\theta; x)],
    \end{align}
    where the set of decision variables $\Theta$ is a closed convex set, e.g., $\Theta \subseteq \mathbb{R}^d$, $\mathcal D$ is a discrepancy measure between distributions, $\emp$ is the empirical data distribution, $l(\theta;x)$ is the loss function given parameter $\theta$ and data point $x$.
    Below, we consider DRO in the two settings of probability metric $\mathcal D$: the $p$-Wasserstein DRO ($p\geq 1$)~\citep{gao_distributionally_2016}
    as well as the Kernel(-MMD) DRO~\citep{zhu_kernel_2021}.
    Our scope is crucially broader than the common Wasserstein DRO reformulations using conic linear duality,
    which often omit the discussion of \emph{non-linear-in-$x$ losses} common in machine learning.
    We propose the following primal-dual reformulation of DRO.
    \begin{lemma}
      [Primal-dual reformulation of Wasserstein and Kernel DRO]
      \label{thm-dro-reform-lemma}
      Suppose the probability metric is chosen to be the optimal transport metric (e.g., $p$-Wasserstein distance).
      Then, the DRO problem~\eqref{eq-dro-ipm} admits the reformulation
      \begin{align}
      \inf_{ \stackrel{\gamma>0,\theta \in \Theta,}{
      f\in\Psi_c}}     
      \sup_{\mu\in\mathcal M} 
      \E_{\mu}{(l(\theta;x)- \gamma f(x))}
      -\frac {\gamma}n\sum_{i=1}^n f^c (x_i)
      + \gamma \epsilon
    .
    \label{eq:wdro-potential-reform}
    \end{align}
    $\Psi_c$ is the set of $c$-concave functions~\citep{santambrogio_optimal_2015} and $f^c(y):=\inf_x c(x,y)-f(x)$ denotes the $c$-transform.

      Suppose the probability metric $\mathcal D$ is the MMD,
      then the DRO problem~\eqref{eq-dro-ipm}
      %  admits the reformulation
      % \begin{multline}
      %   \inf_{\theta \in \Theta, f\in \rkhs}     
      %   \sup_{\mu\in\mathcal M} 
      %   \E_{\mu}{(l(\theta;x) - f(x))}
      %   \\
      %   +
      %   \frac 1N\sum_{i=1}^N f (x_i)
      %   + \epsilon\|f\|_\rkhs
      %   .
      %   % \label{eq-kdro-dual}
      % \end{multline}
      % Furthermore, it 
      is equivalent to the smooth saddle-point problem
      \begin{multline}
        \label{eq:kdro_smooth}
            \inf_{\theta \in \Theta, f\in \rkhs}
            \sup_{\mu\in \mathcal M, h \in \rkhs: \|h\|_{\rkhs}\leq 1}
            \\
            \Biggl( 
              % F(\theta,g,\mu,h) :=
              \frac 1n\sum_{i=1}^n f (x_i) + \epsilon \langle h,f \rangle
               + \E_{ \mu} {(l(\theta;x) - f(x))}
            \Biggr)
                .
        \end{multline}
    \end{lemma}
    % See the appendix for the derivation.
    Note that $c$-concavity of a function $f$ means that $f$ is the $c$-transform of a bounded continuous function.
    Lemma~\ref{thm-dro-reform-lemma} shows that those primal-dual DRO formulations have the \emph{\cvxcnc} structure in $\mu, f$ as in the \pd{MFNE} \eqref{eq-general}. 
    (Although the convexity w.r.t. $\mu$ of \eqref{eq:wdro-potential-reform} is lost along the Wasserstein geodesics~\citep{ambrosio2008gradient}.)
    They are also convex in the learning model parameter $\theta$ when the loss $l$ is, 
    \eqref{eq:kdro_smooth} is concave in the smoothing variable $h$, and
    \eqref{eq:wdro-potential-reform} is trivially convex in the dual variable $\gamma$.
    
    \textbf{Wasserstein barycenter.}
    Another important application is the Wasserstein barycenter problem~\citep{agueh2011barycenters,peyre2018computational,li_continuous_2020,tiapkin_stochastic_2021,kroshnin2019complexity,korotin_continuous_2021,dvurechensky2018decentralize}, which can be formulated as a saddle-point optimization problem.
    \begin{multline}
      \min_{\mu \in \mathcal M} 
      \sum_{i=1}^{n}
      \alpha_i
      \left[ {W_p^p}(\mu, \nu_i) \right] 
      \\
      =
      \min_{\mu \in \mathcal M}
      \sum_{i=1}^{n}
      \alpha_i
      \sup_{f_i \in  \Psi_c}
      \left\{
      \int f_i^c\dd \mu
      +   \int 
      f_i \dd \nu_i
      \right\},
    \end{multline}
    where
    $f_i^c$ again denotes the $c$-transform.
    $\nu_i\in\mathcal M$ are given probability measures,
     $f_i \in\Psi_c$ are the $c$-concave Kantorovich potentials associated with the Wasserstein distance.
     % \pdc{Should we ad some more on the importance of the WB problem}
    
    As we have seen, several ML problems can be formulated in our framework of the infinite-dimensional MFNE problem \eqref{eq-general}, which gives a unified optimization view on these applications. 
    An important role in these formulations is played by the dual function set $\mathcal{F}$. Note that in the aforementioned settings of $p$-Wasserstein ($p\neq 1$) metric, one may parameterize Kantorovich potential functions using tools such as the random Fourier features~\citep{rahimi_random_2007} and ICNN~\citep{amos_input_2017}, 
    such as done by \cite{genevay_stochastic_2016,makkuva_optimal_nodate,korotin_wasserstein-2_2020}.
    Therefore, as a first step towards the theoretical analysis of the MFNE problem and to simplify this analysis, we focus on the RKHS functions (and hence the kernel MMD in the DRO setting) setting in the rest of the paper. We underline that, even in this particular setting, we are not aware of algorithms and their convergence results.
    % More general settings are deferred to future work due to the difficulty in theoretically characterizing the approximation error of deep models.

    \subsection{Gradient Flow in Hilbert and Metric Spaces}
    \input{bg-gf.tex}
    \section{RKHS GRADIENT FLOW FOR MODELING MEASURE OPTIMIZATION DYNAMICS}
    Our starting point is to model the infinite-dimensional continuous-time optimization dynamics of the functional optimization problem~\eqref{eq-fn-opt} using the RKHS gradient flows (RKHSGF)
    \begin{align}
      \partial_t f  = - \mathcal V '( f), &\quad f(0,x)=f^0(x)\in\rkhs,
      \label{eq-kgf}
    \end{align}
    where $\mathcal{V}'(f)$ is Fr\'echet (sub)differential in $\rkhs$.
    The usual properties of this gradient flow, such as existence and uniqueness are provided in the appendix.
    In principle, other function spaces in the literature can also be considered in practice, such as the random Fourier feature functions, single hidden-layer neural networks~\citep{bach2017breaking}, and ICNNs.
    To make the theoretical analysis tractable, we focus in the rest of the paper on the RKHS setting. A more general setting is left for future work.
    
    RKHS has been used in machine learning applications of gradient flows.
    For example, in \cite{arbel_maximum_2019}, the squared RKHS norm was used as the driving energy functional for the Wasserstein gradient flow, rather than its use as the dissipation geometry for the flow as in our formulation.
    % \pdc{Is the previous sentence still correct given that we changed WGF to the Fisher-Rao flow?}
    A few other works such as \cite{chu2020equivalence,chewi_svgd_2020,duncanGeometrySteinVariational2023,liu_stein_2019,korba_kernel_2021} studied the kernelized Wasserstein gradient flow in the context of Stein geometry.
    Those particular cases do not exploit the simplicity of gradient flow structure in the Hilbert space.
    We outline some technical background regarding this gradient flow in the RKHS.

    One important distinction between the functional optimization dynamics in RKHSGF and the measure optimization in MNE optimization dynamics is 
    that we do not need advanced structures such as \emph{generalized geodesic convexity} from the WGF setting (see\eg \cite{ambrosio2008gradient})
    ---
    linearity or regular convexity is sufficient for RKHSGF.
    This can be seen in the following example.
    \begin{example}
    \label{lemma-geod-cvx}
      [Convexity]
      In a 2-Wasserstein space $(\mathcal M, W_2)$,
      linear energy functional
      $\mathcal E_0 (\mu ) = \int V_0 \dd \mu$ for $\mu\in (\mathcal M, W_2)$,
      is non-convex geodesically (Section~\ref{sec:geod-cvx-w2}) if and only if the function $V_0(x)$ is non-convex~\citep{ambrosio2008gradient}.
      On the other hand, 
      in an RKHS $\rkhs$,
      linear energy functional
      $\mathcal E_1 (f ) = \hiprod{V_1}{f} $ for $f\in\rkhs$
      is always convex, regardless of the non-convexity of $V_1$.
    \end{example}
    \begin{example}
      \label{ex:wdro-not-work}
      [Intractability of Wasserstein DRO with nonlinear losses]
      The Wasserstein-DRO problem
      \begin{align}
          \inf_{{\theta \in \Theta}}\sup_{ {W}_p(\mu, \emp) \le \epsilon} \E_\mu[l(\theta; x)]
      \end{align}
      where $l(\theta; x)$ is non-convex non-linear w.r.t. $x$ is intractable.
      As discussed in Example~\ref{lemma-geod-cvx}, despite the linear-in-measure inner problem, the objective function 
      $\E_\mu[l(\theta; x)]$
      is not geodesically convex w.r.t. $\mu$ in the Wasserstein space.
      See Section~\ref{sec:geod-cvx-w2} for more details.
      Hence the popular approaches using linear duality reformulation, used in a large body of literature, are not tractable for this class of problems, unless strong restrictions are imposed.
    In other words, there is still no free lunch despite the lifting to the linear structure.
      The only principled treatment in the non-convex case is via Lagrangian relaxation under non-convex objectives~\citep{sinha_certifying_2020}, which no longer preserves the ambiguity-ball structure.
      In contrast,
      we show convergence guarantees for Kernel-MMD DRO in Section~\ref{sec:dro} without placing restricting assumptions on $l(\theta; x)$ w.r.t. $x$.
      % \pdc{This example seems to me to contradict the comment after Lemma~\ref{thm-dro-reform-lemma}, where we claim that the refurmulation for W-DRO is convex-concave.}
      % \jz{I added a comment there: ``the convexity w.r.t. $\mu$ of \eqref{eq:wdro-potential-reform} is lost along the Wasserstein geodesics.''}
    \end{example}
    % \pdc{Perhaps, the following paragraph should be updated since we don't have WGF explicitly anymore.}
    % Hence, while the MNE problem "lifts" the non-convex optimization problems to the measure spaces, the resulting objective functions as in WGF are not \emph{geodesically convex}. Hence, the resulting WGF does not necessarily converge globally even in the linear case above.
    In contrast, the RKHSGF in our formulation converges under the usual convexity in the Hilbert space.
    % This distinction of convexity has also been exploited in the context of distributionally robust optimization with nonlinear or non-convex (in the uncertain variable $x$) DRO objective functions using kernel methods
    % \cite{zhu_kernel_2021}.
    In the next section, we will use the explicit Euler discretization of this gradient flow, coupled with measure-update step,
    % , results in the \emph{kernel mirror descent} (KMD) step
    % \begin{align}
    % f^{k+1}\in \arg\inf_{f\in\rkhs} 
    % F (\mu^k, f^k) +
    % \langle
    % \partial_f F (\mu^k, f^k)
    % ,
    % f - f^k
    % \rangle_\rkhs
    % +
    % \frac{1}{2\tau}\hnorm{f-f^k}^2
    % .
    % \end{align}
    % \jz{use mirror prox instead of MD}
    % Note that the squared RKHS norm $\hnorm{x-y}^2$ is indeed a \emph{Bregman divergence} with the generating map $\phi(x) = \frac{1}{2}\hnorm{x}^2$ since
      % $\hnorm{x-y}^2 = \frac{1}{2}\hnorm{x}^2 - \frac{1}{2}\hnorm{y}^2 - \hiprod{y}{x-y} $.
    to derive the \emph{primal-dual kernel mirror prox} algorithm, which is the discrete-time counterpart to the interacting gradient flow. 
    % We then use it to construct a primal-dual algorithm for solving \mKneRef.
    \paragraph*{Interacting Gradient Flow Dynamics}
    Previously, \cite{domingo-enrich_mean-field_2020} proposed the Interacting Wasserstein Gradient Flow as the infinite-dimensional continuous-time optimization dynamics for solving the MNE problem~\eqref{eq-mne-intro}.
    We now model a gradient system that couples a gradient flow in the Fisher-Rao geometry and an \emph{RKHS gradient flow} (RKHSGF).
    We term the mean-field dynamics
    \emph{Interacting Gradient Flow}.
    Our perspective is to choose the (dual) function space in \mfneRef as a RKHS {$\mathcal{H}$}~\citep{wendland_scattered_2004,steinwart_support_2008,berlinet_reproducing_2011}.
    That is to say, letting $\mathcal F=\rkhs$ in \mfneRef, we solve a specific variation of MFNE, which is the \emph{Mixed Kernel Nash Equilibrium} (MKNE)
    \begin{equation}
      \begin{aligned}
        \label{eq-mixed-kernel-ne}
    \inf_{f\in \rkhs} 
    \sup_{\mu\in \mathcal M}  
    F (f,\mu).
      \end{aligned}
    \end{equation}
    The gradient flow equations governing the interacting system are
    \begin{align}
      &\partial_t\mu - \mu \cdot F'_\mu(f,\mu)
      +\int \mu \cdot F'_\mu(f,\mu)
      = 0,
      \label{eq-react-eq}
      \\
      &\partial_t f  + F'_f(f,\mu)=0, 
      \label{eq-kgf-2-main}
      \\
      &\quad \mu(0, x) = \mu^0(x)\in \mathcal M,
      \quad f(0,x)=f^0(x)\in\rkhs.
      \nonumber
    \end{align}
    where the derivatives $F'_\mu, F'_f$ are taken in the sense of the Fr\'echet differential w.r.t. $L^2$ and RKHS norms respectively.
    While it is also possible to use other spaces to replace $\rkhs$, such as neural networks to generate so-called neural flows,
    for the sake of analysis, this paper only focuses on the RKHS setting.

    When viewed standalone, the properties of the RKHS gradient flow equation~\eqref{eq-kgf-2-main} 
    and the \emph{reaction equation} \eqref{eq-react-eq} have already been discussed above.
    Our choice of the Fisher-Rao type of flows is due to that its relation to the entropic mirror descent algorithm~\eqref{eq:md-intro}, which
    is a common tool for optimizing over probability measures in practical computation.
    Later, we further provide the convergence guarantee for a discrete-time mirror prox algorithm based on our choice of flows.

    Motivated by the continuous dynamics, we now propose the discrete-time saddle-point optimization algorithm that mimics the continuous-time dynamics and prove its convergence rate bounds.

    \section{A PRIMAL-DUAL KERNEL MIRROR PROX ALGORITHM}
    \input{sec-mirror-prox.tex}
    \section{DISCUSSION}
    % \vspace{-0.3cm}
    In conclusion, we model the measure optimization dynamics as an interacting gradient flow that couples the Fisher-Rao and the RKHS flow. We then analyze its corresponding primal-dual kernel mirror prox algorithm. 
    We provide the first unified convergence analysis for the MKNE problem class. Furthermore, by applying our analysis to the primal-dual reformulation of DRO with kernel MMD constraints,
    we established algorithmic guarantees for DRO under relaxed conditions such as nonlinear losses.
    % , with solutions generated by finitely many steps of the KMP algorithm.
    % As this paper focuses on theoretical analysis, code implementation, which is a combination of the measure and kernel steps (see Remark~\ref{rmk:particle-md}), is left for future work. 
    In the future, we plan to investigate dual functional spaces $\mathcal{F}$ other than RKHS as well as practical implementation and parameterization of the mirror prox algorithm. A related direction of future work is extension of the algorithms and analysis for problems lacking convexity based on techniques in \cite{beznosikov2022decentralized,gorbunov2022clipped}.
    % Another future direction is, instead of the Fisher-Rao type mirror descent, using the mean-field analysis of the Wasserstein gradient flow and Langevin SDE, for which we refer to recent works such as \cite{domingo-enrich_mean-field_2020,chizatSparseOptimizationMeasures2020,nitandaConvexAnalysisMean2022a}.

\subsubsection*{Acknowledgements}
The authors thank Anna Korba and Lénaïc Chizat for the helpful discussion and pointing out relevant references.
This project has received funding from the Deutsche Forschungsgemeinschaft (DFG, German Research Foundation) under Germany's Excellence Strategy – The Berlin Mathematics Research Center MATH+ (EXC-2046/1, project ID: 390685689).

%     \subsection*{\jz{Consider removing}Example: Interacting Wasserstein-Kernel Gradient Flow}
    
% We consider another example that alternatives between a WGF and an \emph{RKHS gradient flow} (RKHSGF).
%     The gradient flow equations are
%     \begin{align}
%       \partial_t\mu - \nabla\cdot( \mu \cdot \nabla F'_\mu(\mu, f))
%       = 0,&\quad \mu(0, x) = \mu^0(x)\in \mathcal M,
%       \label{eq-ct-eq}
%       \\
%       \partial_t f  + F'_f(\mu, f)=0, &\quad f(0,x)=f^0(x)\in\rkhs.
%       \label{eq-kgf-2}
%     \end{align}
%     where $\rkhs$ is an RKHS.
%     The derivatives $F'_\mu, F'_f$ are taken in the sense of the Fr\'echet differential in the corresponding spaces.

%     The properties of the RKHS gradient flow~\eqref{eq-kgf-2} has already been discussed in the main text.
%     The Fokker-Planck equation \eqref{eq-ct-eq} can be viewed either as 
%     a Wasserstein gradient flow of the energy functional
%     $\mathcal E_0(\mu) :=F(\mu, f)$ from the PDE perspective,
%     or as a Langevin SDE,
%     and is already well understood in the machine learning community.
%     We refer to \cite{santambrogio_euclidean_2017,ambrosio2008gradient} for its properties and recent works such as \cite{nitandaConvexAnalysisMean2022a,chizat2018global,chizatSparseOptimizationMeasures2020} for recent machine learning applications.


\bibliography{zot_references,zot_ref_additional}
\bibliographystyle{icml2023}
% \bibliographystyle{authoryear}

%%%%%%%%%%%%%%%%%%%%%%%%%%%%%%%%%%%%%%%%%%%%%%%%%%%%%%%%%%%%
\section*{Checklist}

% % %%% BEGIN INSTRUCTIONS %%%
% The checklist follows the references. For each question, choose your answer from the three possible options: Yes, No, Not Applicable.  You are encouraged to include a justification to your answer, either by referencing the appropriate section of your paper or providing a brief inline description (1-2 sentences). 
% Please do not modify the questions.  Note that the Checklist section does not count towards the page limit. Not including the checklist in the first submission won't result in desk rejection, although in such case we will ask you to upload it during the author response period and include it in camera ready (if accepted).

% \textbf{In your paper, please delete this instructions block and only keep the Checklist section heading above along with the questions/answers below.}
% % %%% END INSTRUCTIONS %%%

 \begin{enumerate}

 \item For all models and algorithms presented, check if you include:
 \begin{enumerate}
   \item A clear description of the mathematical setting, assumptions, algorithm, and/or model. [Yes]
   \item An analysis of the properties and complexity (time, space, sample size) of any algorithm. [Yes]
   \item (Optional) Anonymized source code, with specification of all dependencies, including external libraries. [Not Applicable]
 \end{enumerate}

 \item For any theoretical claim, check if you include:
 \begin{enumerate}
   \item Statements of the full set of assumptions of all theoretical results. [Yes]
   \item Complete proofs of all theoretical results. [Yes]
   \item Clear explanations of any assumptions. [Yes]     
 \end{enumerate}

 \item For all figures and tables that present empirical results, check if you include:
 \begin{enumerate}
   \item The code, data, and instructions needed to reproduce the main experimental results (either in the supplemental material or as a URL). [Not Applicable]
   \item All the training details (e.g., data splits, hyperparameters, how they were chosen). [Not Applicable]
         \item A clear definition of the specific measure or statistics and error bars (e.g., with respect to the random seed after running experiments multiple times). [Not Applicable]
         \item A description of the computing infrastructure used. (e.g., type of GPUs, internal cluster, or cloud provider). [Not Applicable]
 \end{enumerate}

 \item If you are using existing assets (e.g., code, data, models) or curating/releasing new assets, check if you include:
 \begin{enumerate}
   \item Citations of the creator If your work uses existing assets. [Not Applicable]
   \item The license information of the assets, if applicable. [Not Applicable]
   \item New assets either in the supplemental material or as a URL, if applicable. [Not Applicable]
   \item Information about consent from data providers/curators. [Not Applicable]
   \item Discussion of sensible content if applicable, e.g., personally identifiable information or offensive content. [Not Applicable]
 \end{enumerate}

 \item If you used crowdsourcing or conducted research with human subjects, check if you include:
 \begin{enumerate}
   \item The full text of instructions given to participants and screenshots. [Not Applicable]
   \item Descriptions of potential participant risks, with links to Institutional Review Board (IRB) approvals if applicable. [Not Applicable]
   \item The estimated hourly wage paid to participants and the total amount spent on participant compensation. [Not Applicable]
 \end{enumerate}

 \end{enumerate}

%%%%%%%%%%%%%%%%%%%%%%%%%%%%%%%%%%%%%%%%%%%%%%%%%%%%%%%%%%%%%%%%%%%%%%%%%%%%%%%
%%%%%%%%%%%%%%%%%%%%%%%%%%%%%%%%%%%%%%%%%%%%%%%%%%%%%%%%%%%%%%%%%%%%%%%%%%%%%%%
% APPENDIX
%%%%%%%%%%%%%%%%%%%%%%%%%%%%%%%%%%%%%%%%%%%%%%%%%%%%%%%%%%%%%%%%%%%%%%%%%%%%%%%
%%%%%%%%%%%%%%%%%%%%%%%%%%%%%%%%%%%%%%%%%%%%%%%%%%%%%%%%%%%%%%%%%%%%%%%%%%%%%%%
\newpage
\appendix
\onecolumn

\section{FURTHER TECHNICAL BACKGROUND}
\input{apx-tech-detail.tex}
\end{document}

%% file: bg-gf.tex
Recent theoretical analysis of generative models via MNE,
e.g., \cite{hsieh2019finding,domingo-enrich_mean-field_2020,wang_exponentially_2022,trillos_adversarial_2023}, adopted the mean-field point of view closely related to PDE gradient flows of probability measures~\citep{ambrosio2008gradient,otto_double_1996,jordan_variational_1998,otto_geometry_2001}.
Notably, works such as \cite{hsieh2019finding} modeled training dynamics of GAN as sampling using Langevin SDE, which is also equivalent to solving the Fokker-Planck PDE.

Intuitively, a gradient flow describes a dynamical system that is driven towards the fastest dissipation of certain energy.
This system is called a \emph{gradient system}.
For example, the dynamical system described by 
an ordinary differential equation in the Euclidean space
that follows the negative gradient direction of some function $f$, \(\dot x(t) = - \nabla f(x(t)), x(t)\in \mathbb R^d \),
is a simple gradient system.
Of particular interest to this paper in the context of measure optimization ~\eqref{eq-meas-opt-intro} is the gradient flow in the Fisher-Rao metric space,  which is the gradient system that generates the 
following \emph{reaction equation} as its gradient flow equation
\begin{multline}
  \partial_t\mu = - \mu \cdot (\mathcal{E}'(\mu) - \int \mu \cdot \mathcal{E}'(\mu))
  ,\\
  \mu(0, x) = \mu^0(x)\in \mathcal M,
  \label{eq-react-eq-background}
\end{multline}
where $\mathcal{E}'(\mu)$ is the first variation.
Alternatively, one can also view the whole r.h.s. as the Fisher-Rao metric gradient.
More generally, gradient flows in metric spaces \citep{ambrosio2008gradient} have gained significant attention in recent machine learning literature due to the study of Wasserstein gradient flow (WGF), originated from the seminal works of Otto and colleagues, e.g., \cite{otto_double_1996,jordan_variational_1998,otto_geometry_2001}.
Rigorous characterizations of general metric gradient systems have been carried out in PDE literature, for which we refer to \cite{ambrosio2008gradient} for a complete treatment.
% s and \cite{santambrogio_optimal_2015, peletier_variational_2014,mielke_gradients_nodate} for a first principles' introduction.

In practical optimization, this flow
corresponds to
% can be implemented as 
the entropic mirror descent step
% Fisher-Rao gradient flow's close connection with mirror descent can be easily seen by re-writing the reaction equation~\eqref{eq-react-eq-background} in the logarithmic form
%               ${\frac{\dd }{\dd t}\left(\log \mu\right) =  -\mathcal{E}'(\mu)}$ and applying the forward Euler discretization
%     % \begin{align*}
%         $\frac1\tau\left(\log \mu_{k+1} - \log \mu_{k}\right) =  -\mathcal{E}'(\mu_k)$
%         .
%     % \end{align*}
% This equation is equivalent to the mirror descent step using the KL-divergence
\begin{align}
  \label{eq:md-intro}
    \mu_{k+1}\gets \operatorname*{\arg\inf}_{\mu\in\mathcal M} \int \mathcal{E}'(\mu_k) \dd (\mu - \mu_k) + \frac{1}{\tau}D_{\mathrm{KL}}(\mu, \mu_k),
\end{align}
giving multiplicative update
% \begin{align*}
         $\mu_{k+1} = \frac1Z \mu_{k} \cdot e^{-\tau \mathcal{E}'(\mu_k)}$,
         where $Z$ is a normalization constant.

%% file: sec-mirror-prox.tex
\label{S:MP}

To construct our mirror prox algorithm, following \cite{hsieh2019finding}, we make some mild regularity assumptions.
As noted by those authors, the following assumptions can be met by most practical applications.
% \footnote{\pd{As noted in \cite{hsieh2019finding}, these assumptions are met by most practical applications. We also believe that these assumptions may be relaxed based on the techniques in \cite{aubin2022mirror}}}.
% we consider generic variable $x \in \mathbb{R}^p$ with domain $\mathcal{X}$ and  make some mild regularity assumptions. 
Namely, we restrict $ \mathcal M$ to the set of probability measures on {$\bar{\Omega}$} that admit densities w.r.t. the Lebesgue measure and have a density that is continuous and positive almost everywhere on {$\bar{\Omega}$}.
We note that this assumption is only made for the ease of the analysis. In practice, there are also numerous particle-based algorithms that can perform the mirror descent step, see Remark~\ref{rmk:particle-md}.
We also assume that there is a Hilbert space $\rkhs$ and a convex and closed set $H \subseteq \rkhs$.
For the sake of generality, we consider a slight variation of \mKneRef as the following general infinite-dimensional saddle-point problem on the spaces of measures and functions
\begin{align}
\label{eq:SPP_general}
    \inf_{f  \in H \subseteq \rkhs}
    \sup_{\mu\in \mathcal M}
    \quad F(f,\mu).
\end{align}
For shortness, we denote the set of all variables by $u:=(f,\mu)$ and slightly abuse notation by defining $F(u):=F(f,\mu)$. We consider here the setting of two variables only for simplicity. An extension for a more general problem formulation covering the DRO problem \eqref{eq:kdro_smooth} is provided in the appendix. Moreover, in the next section, we consider DRO problem \eqref{eq:kdro_smooth} as a particular case study and provide technical details to check that the main assumptions of this section hold for that problem.
Our first main assumption in this section is as follows.
\begin{assumption}
\label{Asm:SPP_gen_convexity}
The functional $F( f,\mu )$ is convex in $f$ for fixed $\mu$ and concave in $\mu$ for fixed $f$. %Further, we assume that the objective $F$ is Fr\'echet differentiable w.r.t. each variable.
\end{assumption}

\subsection{Preliminaries}
To construct the mirror prox algorithm for problem \eqref{eq:SPP_general} we need, first, to introduce proximal setup, which consists of norms, their dual, and Bregman divergences over the space of each of the variables. 

% For the space of the variable $\theta$, we introduce the standard proximal setup with the self-dual Euclidean norm $\|\cdot \|_2$,  distance-generating function $d_{\theta}(\theta) = \frac{1}{2}\|\theta\|_2^2$, which gives Bregman divergence $B_{\theta}(\theta,\Breve{\theta})=\frac{1}{2}\|\theta-\Breve{\theta}\|_2^2$. This leads to the mirror step defined as
% \begin{equation}
%     \theta_+ = \MD^{\theta,\Theta}_{\eta}(\theta, \xi_\theta) = \arg \min_{\tilde{\theta} \in \Theta} \{ \langle \tilde{\theta}, \eta\xi_\theta \rangle + \frac{1}{2}\|\tilde{\theta} - \theta\|_2^2\}.
% \end{equation}
% We note that our choice of the Euclidean proximal setup is made for simplicity and that other standard proximal setups are possible \cite{nemirovski_robust_2009}.

For the space of the variable $f$, we use the norm of the Hilbert space $\|\cdot\|_{\rkhs}$, distance generating function $d_{f}(f) = \frac{1}{2}\|f\|_{\rkhs}^2$, which gives Bregman divergence $B_{\rkhs}(f,\Breve{f})=\frac{1}{2}\|f-\Breve{f}\|_{\rkhs}^2$. This leads to the mirror descent step, given step size $\eta > 0$ and direction $\xi_f \in \rkhs$,
\begin{equation}
    f_+ = \MD^{f,H}_{\eta}(\pd{f_0},  \xi_f) = \arg \min_{\tilde{f} \in H} \{ \langle \tilde{f}, \eta \xi_f \rangle + \frac{1}{2}\|\tilde{f} - \pd{f_0}\|_{\rkhs}^2\}.
    \label{eq:md-step-rkhs}
\end{equation}
Note that the explicit expression for the above step is the projection of $\pd{f_0}-\eta \xi_f$ onto the set $H$ w.r.t. the RKHS norm.

For the space of the variable $\mu$, we follow \cite{hsieh2019finding} and, first, introduce the Total Variation (TV) norm for the elements of $\mathcal{M}$
$\| \mu \|_{TV} = \sup_{\|\xi\|_{L^{\infty}}\leq 1} \int \xi d\mu = \sup_{\|\xi\|_{L^{\infty}}\leq 1} \la \xi ,\mu\ra$,
where $\|\xi\|_{L^{\infty}}$ is the $L^{\infty}$-norm of functions. 
To define the mirror step, we use (negative) Shannon entropy and its Fenchel dual defined respectively as 
\begin{equation}
    \Phi(\mu) = \int d\mu \ln \frac{d\mu}{dx}, \quad \Phi^*(\xi) =  \ln \int e^\xi dx
\end{equation}
% \begin{equation}
%     
% \end{equation}
defined for $\xi$ from the space $\mathcal{M}^*$ of all bounded integrable functions on {$\bar{\Omega}$}. The corresponding Bregman divergence is the relative entropy given by  
\begin{equation}
    D_\Phi(\mu,\breve{\mu}) =  \int d\mu \ln \frac{d\mu}{ d\breve{\mu}}.
\end{equation}
As discussed previously in \eqref{eq:md-intro},
this leads to the mirror step
\begin{multline}
\label{eq-mirror-mu}
    \mu_+ = \MD^{\mu}_{\eta}(\pd{\mu_0},  \xi_\mu) 
    = 
    \arg \min_{\tilde{\mu} \in \mathcal M} \{ \langle \tilde{\mu}, \eta \xi_\mu \rangle + D_\Phi(\tilde\mu,{\pd{\mu_0}})
    \}.
    \\
    = d \Phi^*(d \Phi(\pd{\mu_0}) - \eta \xi_\mu)
    % \\
    \equiv  \quad d\mu_+ 
    = \frac{e^{-\eta \xi_\mu}d\pd{\mu_0}}{\int e^{-\eta \xi_\mu}d\pd{\mu_0}}.
\end{multline}

% Finally, for the space of the variable $h$, we do the same as for the variable $f$. Namely, we use the distance generating function $d_{h}(h) = \frac{1}{2}\|h\|_{\rkhs_h}^2$, which gives Bregman divergence $B_{\rkhs_h}(h,\Breve{h})=\frac{1}{2}\|h-\Breve{h}\|_{\rkhs_h}^2$. This leads to the mirror step
% \begin{equation}
%     h_+ = \MD^{h,H}_{\eta}(h,  \xi_h) = \arg \min_{\tilde{h}\in H} \{ \langle \tilde{h}, \eta \xi_h \rangle + \frac{1}{2}\|\tilde{h} - h\|_{\rkhs_h}^2\} .
% \end{equation}

Our second main assumption is as follows.
\begin{assumption}
\label{Asm:SPP_gen_smooth}
The functional $\pd{F(u):=}F(f,\mu)$  is Fr\'echet differentiable in $f$ w.r.t. the RKHS norm and in $\mu$, in $L^2$, and the derivatives \pd{$F'_f(u):=F'_f(f,\mu)$, $F'_{\mu}(u):=F'_{\mu}(f,\mu)$} are Lipschitz continuous in the following sense 
\begin{align}
%&\|F'_\theta(u)-F'_\theta(\tilde{u})\|_2 \leq L_{\theta\theta}\|\theta-\tilde{\theta}\|_2 +  L_{\theta f}\|f-\tilde{f}\|_{\mathcal{H}_f} +  L_{\theta\mu}\|\mu-\tilde{\mu}\|_{TV} +   L_{\theta h}\|h-\tilde{h}\|_{\mathcal{H}_h}, \\
&\|F'_f(u)-F'_f(\tilde{u})\|_{\mathcal{H}} \leq L_{f f}\|f-\tilde{f}\|_{\mathcal{H}} +  L_{f\mu}\|\mu-\tilde{\mu}\|_{TV} , \\
&\|F'_\mu(u)-F'_\mu(\tilde{u})\|_{L^{\infty}} \leq  L_{\mu f}\|f-\tilde{f}\|_{\mathcal{H}} +  L_{\mu\mu}\|\mu-\tilde{\mu}\|_{TV} .  
%&\|F'_h(u)-F'_h(\tilde{u})\|_{\mathcal{H}_h} \leq L_{h\theta}\|\theta-\tilde{\theta}\|_2 +  L_{h f}\|f-\tilde{f}\|_{\mathcal{H}_f} +  L_{h\mu}\|\mu-\tilde{\mu}\|_{TV} +   L_{h h}\|h-\tilde{h}\|_{\mathcal{H}_h}.
\end{align}
\end{assumption}
We also denote
\begin{equation}
\label{eq:L_general}
    L = \max\{L_{f,f},L_{f,\mu},L_{\mu,f},L_{\mu,\mu}\}.
    % \max_{\kappa_1,\kappa_2\in \{f,\mu\}}\{L_{\kappa_1\kappa_2} \}.
\end{equation}

\subsection{Kernel Mirror Prox Algorithm and Its Analysis}

The updates of the general  infinite-dimensional mirror prox algorithm for problem \eqref{eq:SPP_general} are given in Algorithm \ref{alg:mirror_prox_gen}. \pd{Recall that, for shortness, we use the notation $u$ for the pair $(f,\mu)$ and slightly abuse notation by setting $F(u):=F(f,\mu)$.}
\begin{algorithm}[ht!]
\caption{General Mirror Prox}
{
\begin{algorithmic}[1]\label{alg:mirror_prox_gen}
    \REQUIRE{ Initial guess ${\tilde{u}_0:=}( \tilde{f}_0,\tilde{\mu}_0 )$, step-sizes $ \eta_f,\eta_\mu  >0$.}
    %such that $\ds\bt = [t_1^\top, \ldots, t_m^\top]^\top$, $\ds t_i = \argmin_{t\in\sT_i} d_{\sT_i}(t)$ for each $t_i\in\braces{x_i, p_i, y_i, q_i, z_i, s_i}$ and corresponding set $\sT_i\in\braces{\bar\sX, \sP_i, \bar\sY, \sQ_i, \R^{d_y}, \R^{d_x}}$, $i = 1, \ldots, m$.}
    \FOR{$k = 0, 1, \ldots, N - 1$}
        \STATE{Compute
        \begin{align*}
            % $$ 
                \hspace{-3em}f_{k} = \MD^{f,H}_{\eta_f}(\tilde{f}_k, F'_f(\tilde{u}_k)), 
                %   \qquad 
                \mu_{k} = \MD^{\mu}_{\eta_\mu}(\tilde{\mu}_k, -F'_\mu(\tilde{u}_k)).
                % $$ 
            \end{align*}
            }
        \STATE{Compute
            % $$ 
            \begin{align*}
                \hspace{-3em}\tilde{f}_{k+1} = \MD^{f,H}_{\eta_f}(\tilde{f}_k,F'_f(u_k)), 
                \tilde{\mu}_{k+1} = \MD^{\mu}_{\eta_\mu}(\tilde{\mu}_k, -F'_\mu(u_k)).
            \end{align*}
            % $$ 
            }
        % \STATE \pd{$k=k+1$. \color{red}(you already had it in the for loop statment)}
    \ENDFOR
    \STATE Compute $( \bar{f}_N,\bar{\mu}_N)=\bar{u}_N = \frac{1}{N}\sum_{k=0}^{N-1} u_k$.
    % \ENSURE{For $\bt\in\braces{\bx, \bp, \by, \bq, \bs, \bz}$ compute $\ds\hat \bt^{N} = \frac{1}{N}\sum_{k=0}^{N-1} \bt^{k+\frac{1}{2}}$.}
\end{algorithmic}
}
\end{algorithm}
For its analysis we need the following auxiliary results. The first one is used for the mirror steps applied to the variable $f$.
% The proof is deferred to the Appendix.
\begin{lemma}
\label{Lm:rkhs_step}
Let $\mathcal H$ be a Hilbert space and let $H\subset \mathcal H$ be convex and closed. Let $\tilde{h} \in H$ and $\xi, \tilde{\xi} \in \mathcal H^*=\mathcal H$, $\eta >0$, and 
\begin{align*}
    h&= \arg \min_{\hat{h}\in H} \left\{\la \hat{h}, \eta \xi   \ra + \frac{1}{2}\|\tilde{h}-\hat{h}\|_{\mathcal{H}}^2\right\} = \MD^{h,H}_{\eta}(\tilde{h},  \xi), \\%\label{eq:Lm:rkhs_step_1}\\ 
    \tilde{h}_+&= \arg \min_{\hat{h}\in H} \left\{\la \hat{h}, \eta \tilde{\xi}  \ra + \frac{1}{2}\|\tilde{h}-\hat{h}\|_{\mathcal{H}}^2\right\} = \MD^{h,H}_{\eta}(\tilde{h},  \tilde{\xi}). %\label{eq:Lm:rkhs_step_2}
\end{align*}
Then, for any $\hat{h} \in H$,
\begin{multline*}
    \la h - \hat{h}, \eta \tilde{\xi} \ra \leq 
    \frac{1}{2}\|\hat{h}-\tilde{h}\|_{\mathcal{H}}^2 
    - \frac{1}{2}\|\hat{h}-\tilde{h}_+\|_{\mathcal{H}}^2 
    \\
    + \frac{\eta^2}{2}\|\tilde{\xi}-\xi\|_{\mathcal{H}}^2 
    - \frac{1}{2}\|h-\tilde{h}\|_{\mathcal{H}}^2.
\end{multline*}
\end{lemma}

The second result characterizes the mirror step with respect to the measure $\mu$. 

\begin{lemma}[{\citep[Lemma 5]{hsieh2019finding}}]
\label{Lm:measure_step}
Let $\tilde{\mu} \in \mathcal M$ and $\xi, \tilde{\xi} \in  \mathcal{M}^*$, $\eta >0$, and 
\begin{align}
    \mu= \MD^{\mu}_{\eta}(\tilde{\mu},  \xi),  \qquad
    \tilde{\mu}_+= \MD^{\mu}_{\eta}(\tilde{\mu},  \tilde{\xi}).
\end{align}
Then, for any $\hat{\mu} \in \mathcal M$
\begin{multline}
    \la \mu - \hat{\mu}, \eta \tilde{\xi} \ra \leq 
    D_{\Phi}(\hat{\mu},\tilde{\mu}) 
    - D_{\Phi}(\hat{\mu},\tilde{\mu}_+)
    \\
    + \frac{\eta^2}{8}\|\tilde{\xi}-\xi\|_{L^{\infty}}^2 
    - 2\|\mu-\tilde{\mu}\|_{TV}^2.
\end{multline}
\end{lemma}

The following result gives the convergence rate of Algorithm \ref{alg:mirror_prox_gen}.
% The proof is deferred to the Appendix.
\begin{theorem} 
\label{Th:ideal_MP_determ}
Let Assumptions \ref{Asm:SPP_gen_convexity}, \ref{Asm:SPP_gen_smooth} hold. Let also the step sizes in Algorithm \ref{alg:mirror_prox_gen} satisfy $ \eta_f=\eta_\mu = \frac{1}{16L}$, where $L$ is defined in \eqref{eq:L_general}. Then, for any compact set $U =   U_f \times U_\mu   \subseteq  H \times \mathcal{M} $, the sequence $( \bar{f}_N,\bar{\mu}_N)$ generated by Algorithm \ref{alg:mirror_prox_gen} satisfies
\begin{multline*}
\max_{\mu \in U_\mu  } F(\bar{f}_N,\mu) - \min_{f \in U_f}F( f,\bar{\mu}_N ) 
\\
\leq \frac{8L}{N}\max_{u \in U} \left(  \|f-\tilde{f}_0\|_{\rkhs}^2 + 2 D_{\Phi}(\mu,\tilde{\mu}_0)   \right).
\end{multline*} 
\end{theorem}
% \jz{@Pavel: this seems be a typo? PD gap of saddle point should be:
% \begin{align*}
% % &
% \max_{\mu \in U_\mu }F(\bar{f}_N,\mu) 
% -
% \min_{ f \in U_f} F( f,\bar{\mu}_N )
% % \leq \frac{8L}{N}\max_{u \in U} \left(  \|f-\tilde{f}_0\|_{\rkhs}^2 + 2 D_{\Phi}(\mu,\tilde{\mu}_0)   \right).
% \end{align*}
% Same for the stochastic case below
% }

\begin{remark}
    [Implementation of mirror descent steps]
Algorithm~\ref{alg:mirror_prox_gen} requires two mirror steps in the functional variables $\mu, f$, both have already been separately studied in the machine learning literature.
The $f$-update, optimization w.r.t. RKHS functions, has already been implemented for various learning tasks, e.g., by \cite{dai2014scalable,genevay_stochastic_2016,tiapkin_stochastic_2021,zhu_kernel_2021,pmlr-v202-kremer23a}.
For example,
step~\eqref{eq:md-step-rkhs}
can be implemented as
solving convex quadratic program
when using the RKHS function parameterization in the form $\tilde f =\sum_{j}\alpha_j k(x_j, \cdot)$.
See those references for more implementation details.
For the $\mu$-update,
many algorithms have been proposed including many particle-update schemes such as \cite{mei_mean_2018,wang_convergence_2021,chizat2018global,dai2016provable,hsieh2019finding,kimSymmetricMeanfieldLangevin2024}, as well as natural gradient descent~\cite{amariNaturalGradientWorks1998,khan2018fast}.
% Note that each of this implementation comes with its tailored error analysis.
An orthogonal direction is to incorporate, e.g., based on inexact Mirror Prox schemes \cite{stonyakin2021inexact,stonyakin2022generalized}, an additional error estimates incurred by adopting one of those specific implementation schemes, e.g., particle mirror descent, which is beyond the scope of this paper and left for future work. Our current scope is also similar to another recent work on infinite-dimensional mirror descent analysis by \cite{aubin2022mirror}.
\label{rmk:particle-md}
\end{remark}

\subsection{Analysis of Stochastic Kernel Mirror Prox}

To account for potential inexactness in the first-order information, we assume that, instead of exact derivatives, the algorithm uses their inexact counterparts $\tilde{F}'_f(u),\tilde{F}'_\mu(u)$ that may be random and are assumed to satisfy the following assumption.
\begin{assumption}
\label{Asm:SPP_gen_stoch}
\begin{align*}
&F'_f(u)=\E\tilde{F}'_f(u), \quad F'_\mu(u)=\E\tilde{F}'_\mu(u), \\
&\E\|F'_f(u)-\tilde{F}'_f(u)\|_{\mathcal{H}}^2 \leq \sigma^2_f, 
\\
&
\E\|F'_\mu(u)-\tilde{F}'_\mu(u)\|_{L^{\infty}}^2 \leq \sigma^2_\mu.
\end{align*}
\end{assumption}

\begin{theorem}
\label{Th:ideal_MP_stoch}
Let Assumptions \ref{Asm:SPP_gen_convexity}, \ref{Asm:SPP_gen_smooth},  \ref{Asm:SPP_gen_stoch} hold. Let also in Algorithm \ref{alg:mirror_prox_gen} the stochastic derivatives be used instead of the deterministic and the step sizes satisfy $\eta_f=\eta_\mu=\frac{1}{16L}$, where $L$ is defined in \eqref{eq:L_general}. Then, for any compact set $U =  U_f \times U_\mu  \subseteq  H \times \mathcal{M} $, the sequence $( \bar{f}_N,\bar{\mu}_N )$ generated by Algorithm \ref{alg:mirror_prox_gen} satisfies 
\begin{multline*}
\E \left\{\max_{\mu \in U_\mu  } F(\bar{f}_N,\mu) - \min_{f \in U_f}F( f,\bar{\mu}_N ) \right\}
\\
\leq \frac{8L}{N}\max_{u \in U} \left(\|f-\tilde{f}_0\|_{\mathcal{H}}^2 + 2D_{\Phi}(\mu,\tilde{\mu}_0) \right) + \frac{3(\sigma^2_f+\sigma^2_\mu)}{16L} .
\end{multline*} 
\end{theorem}
% The proof is deferred to the Appendix.
Let us denote $\sigma^2= \sigma^2_f+\sigma^2_\mu $. 
Theorem \ref{Th:ideal_MP_stoch} guarantees the same convergence rate as in the exact case, but up to some vicinity which is governed by the level of noise. In most cases, the $\sigma^2/L$ term can be made of the same order $1/N$ by using the mini-batching technique. Indeed, a mini-batch of size $N$ allows us to reduce the variance from $\sigma^2$ to $\sigma^2/N$. Yet, we note that in this case, $N$ iterations will require the number of samples $O(N^2)$. This corresponds to iteration complexity $O(1/\varepsilon)$ and sample complexity $O(1/\varepsilon^2)$ to reach an accuracy $\varepsilon$.

An alternative would be to use the information about the diameter of the set $U$. Indeed, assume that 
\[
\max_{u \in U} \left(  \|f-\tilde{f}_0\|_{\mathcal{H}}^2 + 2D_{\Phi}(\mu,\tilde{\mu}_0)  \right) \leq \Omega_{U}^2.
\]
% Then, we obtain the following counterpart of the r.h.s. of \eqref{eq:Th:ideal_MP_stoch_proof_4}
% \jz{this is in the appendix?}
% choosing $\eta_f=\eta_\mu=\eta$:
% $
% \frac{\Omega_{U}^2}{2N \eta} + 3\sigma^2\eta.
% $
Fixing the number of steps $N$ and choosing
$
 \eta_f=\eta_\mu= \min\left\{ \frac{1}{16L}, \frac{\Omega_{U} \sigma }{ \sqrt{6N}} \right\},
$
we obtain the following result
\begin{multline}
 \E \left\{\max_{\mu \in U_\mu  } F(\bar{f}_N,\mu) - \min_{f \in U_f}F( f,\bar{\mu}_N ) \right\} 
\\
\leq \max \left \{\frac{8L\Omega_{U}^2}{N}, \sqrt{\frac{3\sigma^2 \Omega_{U}^2}{2N}}\right\}.
\end{multline}

\section{DRO Algorithmic Guarantees using Kernel Mirror Prox}
\label{sec:dro}
In this section we particularize the elements of KMP in Algorithm \ref{alg:mirror_prox_gen} for the specific DRO problem \eqref{eq-dro-ipm} equipped with the kernel MMD constraint using its primal-dual formulation \eqref{eq:kdro_smooth}.
% \begin{multline*}
%     \inf_{\theta \in {\Theta}, f\in \rkhs}
%     \sup_{\mu\in \mathcal M, h \in \rkhs: \|h\|_{\rkhs}\leq 1}
%     \\
%     \Biggl( 
%       % F(\theta,g,\mu,h) :=
%       \frac 1N\sum_{i=1}^N f (x_i) + \epsilon \langle h,f \rangle
%        + \E_{ \mu} {(l(\theta;x) - f(x))}
%     \Biggr)
%         .
% \end{multline*}
In contrast with many reformulation techniques using simple linear duality formulation, we propose a principled  \emph{primal-dual convergence analysis} for DRO using our MKNE framework.
Notably, we prove the finite-sample robustness \emph{algorithmic guarantee} for the DRO solution generated  by the KMP algorithm after $N$ iterations under non-linear DRO losses.
% Previous bounds via concentration arguments do not consider such optimizaiton error

Compared to problem \eqref{eq:SPP_general}, the DRO formulation~\eqref{eq:kdro_smooth} has two additional variables $\theta \in \mathbb{R}^d$ and $h \in H \subset \rkhs$. For these variables, the proximal setup is introduced in the same way as for the variable $f$. We choose $\rkhs$ to be a reproducing kernel Hilbert space with kernel $k$.
Our main assumptions for problem \eqref{eq:kdro_smooth} are 
% save space, can use enum
% \begin{enumerate}[noitemsep]
    % \item \pd{$l(\theta;x)$ is convex w.r.t. $\theta$ for all $x$.}
    % \item 
    \textbf{1).} $L_0=\sup_{x,\theta}\|\nabla_{\theta} l(\theta;x)\|_2 < +\infty$. 
    % \item 
    \textbf{2).} $\nabla_{\theta} l(\theta;x)$ is $L(x)$-Lipschitz w.r.t. $\theta$ and $L_1=\sup_{\mu}\E_{x \sim \mu} L(x)^2 < +\infty$.
    % \item 
    \textbf{3).} $C=\sup _{x} k(x, x)<+\infty$. 
    % \item 
    \textbf{4).} For any fixed $x$, $l(\theta;x)$ is convex w.r.t. $\theta$ (but not necessarily w.r.t. $x$).
% \end{enumerate}
\begin{remark}
Most importantly, we do not place restrictive assumptions on the loss $l$ w.r.t. variable $x$, which is a common practice in the DRO literature. That is, we do not assume that $l(\theta; x)$ is convex, concave, affine, or quadratic in $x$. This is a significant improvement over the previous DRO analysis, see~\cite{netessine_wasserstein_2019}. 
We also note that the convexity assumption w.r.t. $\theta$ is only used to obtain the global convergence guarantee. When it does not hold in practice, we can still execute our primal-dual KMP algorithm for DRO. 
This is not possible with other existing Wasserstein or kernel DRO algorithms.
\end{remark}
We provide more details regarding the set-up of the KMP algorithm for the DRO setting in the appendix.
% Therefore, we obtain $O(1/N)$ convergence rate in the deterministic case and $O(1/\sqrt{N})$ convergence rate in the stochastic case for solving DRO with kernel mirror prox. 
% To be more specific, we establish the following guarantees in the DRO context, which are the first general non-asymptotic convergence guarantees without strong assumptions on $l(\theta; x)$ in terms of its dependence on $x$.
We now give the main result of this section that upper-bounds the solution (sub-)optimality of the KMP algorithm.
To the best of our knowledge, this is the first general non-asymptotic algorithmic guarantees without strong assumptions on $l(\theta; x)$ in terms of its dependence on $x$.
As a reminder, $\emp$ is the empirical data distribution, whose samples are generated from the (unknown) distribution $\mu_0$.
\begin{proposition}
[DRO Guarantee for KMP decision sub-optimality]
Suppose $\THKMP=\frac1N\sum_{k=0}^{N-1}\theta_k$ is the averaged solution produced by the kernel mirror prox algorithm after $N$ steps.
Then, $\forall \epsilon>0$,
the DRO risk associated with the decision $\THKMP$
is bounded by
\begin{multline}
    \sup_{\mmd(\mu, \emp) \le \epsilon} \E_\mu l(\THKMP; x) -
    \\
     \underbrace{\inf_{\theta\in\Theta}\sup_{\mmd(\mu, \emp) \le \epsilon} \E_\mu l(\theta; x)}_{\textrm{(Optimal DRO risk)}}
     \leq {O}\left(\frac{1}{N}\right).
\end{multline}
\label{thm:dro-subopt-guarantee}
\end{proposition}
% \begin{remark}
    \vspace{-0.6cm}
    The bound becomes $O(\frac1{\sqrt{N}})$ in expectation in the stochastic case, which we provide in the appendix.
% \end{remark}
The result above quantifies that the decision generated by the KMP algorithm after $N$ iterations is close to the 
ideal true DRO solution \THDRO in terms of the risk.
This distinguishes our analysis from the ideal DRO bounds that
do not incorporate optimization errors.
% cannot be realized by any optimization algorithms.
The proof of this result, given in the appendix, is due to our analysis of the KMP algorithm in the previous section.

Using this main result, we can derive further guarantees such as the ones below.
Let $\epsilon_n(\delta):= \sqrt{\frac{C}{n}} + \sqrt{\frac{2C\log(1/\delta)}{n}}$ for $C:=\sup_x|k(x,x)|$, which can be calculated for commonly-used kernels\eg $C=1$ for the Gaussian kernel.
Note that $n$ is the sample size instead of the number of iterations $N$.
\begin{corollary}
    [Robustness guarantee for population risk]
        \label{thm:rob-mu0}
    Suppose the ambiguity level $\epsilon$ is chosen such that $\epsilon > \epsilon_n(\delta)$.
        Then,
        with large probability,
        the population risk of the decision $\THKMP$ output by the KMP algorithm after $N$ steps is estimated from above by
    \begin{multline}
            \E_{\mu_0}[l(\THKMP; x)]
        -
        \inf_{\theta\in\Theta}\sup_{\mmd(\mu, \emp) \le \epsilon} \E_\mu l(\theta; x)
        % \\
        \leq 
        O\left(\frac1N\right)
       .
    \end{multline}
\end{corollary}

\begin{corollary}
    [Generalization guarantee under distribution shift]
    The risk under the worst-case distribution shift from the population distribution $\mu_0$, of any radius $\epsilon>0$,
    of the decision $\THKMP$ output by the KMP algorithm after $N$ steps
    is upper-bounded with large probability by
    \begin{multline}
        % \mathbb E 
        \sup_{\mmd(\mu, \mu_0) \le \epsilon}\E_\mu l(\THKMP; x)
        \\
        -
        \inf_{\theta\in\Theta}\sup_{\mmd(\mu, \emp) \le \epsilon} \E_\mu l(\theta; x)
        % \mathbb E Z_{\epsilon, \emp}(\THDRO)
        \leq 
        % \\
        % +
        O\left(\max\left\{\frac{1}{ N}, \frac{1}{\sqrt n}\right\}\right)
       .
    \end{multline}
    \label{thm:rob-d-shift-sup}
\end{corollary}

%% file: apx-tech-detail.tex
\subsection{List of Acronyms}

\begin{itemize}
    \item DRO -- Distributionally Robust Optimization
    \item \eviL -- evolutionary variational inequality
    \item GAN -- Generative Adversarial Network
    \item ICNN -- input convex neural network
    \item IGM -- implicit generative models
    \item IPM -- integral probability metric
    \item KMP -- Kernel Mirror Prox
    \item MMD -- maximum mean discrepancy
    \item MMS -- minimizing movement scheme
    \item MNE -- Mixed Nash Equilibrium
    \item MFNE -- Mixed Functional Nash Equilibrium
    \item MKNE -- Mixed Kernel Nash Equilibrium
    \item RKHS -- Reproducing kernel Hilbert space
    \item RKHSGF -- RKHS gradient flow 
    \item TV -- Total Variation
    \item WGF -- Wasserstein gradient flow
\end{itemize}

\subsection{Proof of Lemma~\ref{thm-dro-reform-lemma}}
We now prove Lemma~\ref{thm-dro-reform-lemma}, of which a more detailed version is stated below.
% \begin{theorem}
%     % \label{thm:kdro-pd}
%   The DRO problem~\eqref{eq-dro-ipm} with IPM above admits the following reformulation,
%   \begin{align}
%     \sup_{\mu\in\mathcal P} \inf_{f\in \rkhs, \tau>0}
%     \E_{\mu}{(l- f)} +  \E_{\emp}g + \frac{\tau}{2}\|f\|^2_\rkhs +\frac{\epsilon^2}{2\tau}
%   \end{align}
%   Furthermore, an optimal choice of the dual variable $\tau$ (temperature / discretization step if viewed as time-incremental minimization)
%   yields
%   \begin{align}
%     \sup_{\mu\in\mathcal P} \inf_{f\in \rkhs}     
%     \E_{\mu}{(l- f)}+
%     \E_{\emp}f
%     + \epsilon\|g\|_\rkhs
%     .
%   \end{align}
%   \end{theorem}
  \begin{lemma}
    [Primal-dual reformulation of Wasserstein and kernel DRO]
    % \label{thm-dro-reform-lemma}
    Suppose the probability metric $\mathcal D$ is chosen to be the MMD,
    then the DRO problem~\eqref{eq-dro-ipm} admits the following equivalent primal-dual reformulations
    \begin{align}
        \inf_{\theta \in {\Theta \subseteq } \mathbb{R}^d,f\in \rkhs, \tau>0}
        \sup_{\mu\in\mathcal M} 
    \E_{\mu}{(l(\theta;x) - f(x))} +  \frac 1n\sum_{i=1}^n f (x_i) + \frac{\tau}{2}\|f\|^2_\rkhs +\frac{\epsilon^2}{2\tau},
    \label{eq-kdro-dual-ap-1}
        \\
      \inf_{\theta \in {\Theta \subseteq } \mathbb{R}^d, f\in \rkhs}     
      \sup_{\mu\in\mathcal M} 
      \E_{\mu}{(l(\theta;x) - f(x))}+
      \frac 1n\sum_{i=1}^n f (x_i)
      + \epsilon\|f\|_\rkhs
      .
      \label{eq-kdro-dual-ap-2}
    \end{align}
    Furthermore, it is equivalent to the smooth saddle-point problem
    \begin{align}
      \label{eq:kdro_smooth-ap}
          \inf_{\theta \in {\Theta \subseteq } \mathbb{R}^d, f\in \rkhs}
          \sup_{\mu\in \mathcal M, h \in \rkhs: \|h\|_{\rkhs}\leq 1}
          \quad
          \left\{ 
            % F(\theta,g,\mu,h) :=
            \frac 1n\sum_{i=1}^n f (x_i) + \epsilon \langle h,f \rangle + \E_{ \mu} {(l(\theta;x) - f(x))}
          \right\}
              .
      \end{align}
      
      Suppose the probability metric $\mathcal D$ is chosen to be the optimal transport metric, e.g., $p$-Wasserstein distance.
      Then, the DRO problem~\eqref{eq-dro-ipm} admits the following equivalent reformulations
      \begin{align}
        \inf_{\gamma>0,\theta \in {\Theta \subseteq } \mathbb{R}^d,
        f\in\Psi_{\gamma\cdot c}}     
        \sup_{\mu\in\mathcal M} 
        \mathbb E_{\mu}{(l(\theta;x) - f(x))}
        - \frac 1n\sum_{i=1}^n  f^{\gamma\cdot c}(x_i)
        + \gamma\cdot\epsilon
      ,
    \\
      \inf_{\gamma>0,\theta \in {\Theta \subseteq } \mathbb{R}^d,
      f\in\Psi_c}     
      \sup_{\mu\in\mathcal M} 
      \E_{\mu}{(l(\theta;x) - \gamma\cdot f(x))}
      -\frac 1n\sum_{i=1}^n \gamma\cdot f^c (x_i)
      + \gamma\cdot\epsilon
    .
    % \label{eq:wdro-potential-reform}
    \end{align}
    Here $\Psi_c$ denotes the set of $c$-concave functions~\citep{santambrogio_optimal_2015} and $f^c(y):=\inf_x c(x,y)-f(x)$ denotes the $c$-transform.
  \end{lemma}

  \begin{proof}
  For this proof, it suffices to consider the case where $\theta$ is fixed, since only the inner maximization is reformulated using duality.
  We first prove the result for the MMD setting.

  \paragraph*{MMD setting.}
  We consider the kernel mean embedding map as a linear constraint
  \begin{align}
      \int {\phi(x)} \dd \mu(x) = h,
  \end{align}
  where $h$ is a function in $\rkhs$
  and $\phi(x):=k(x,\cdot)$ is a canonical feature function of $\rkhs$.
  Using the Lagrange multipliers $f \in \mathcal{H}$ for the above linear constraint and $1/(2\tau)$ for the inequality constraint $\mathrm{MMD}(\mu,\hat{\mu})^2 \leq \epsilon^2$, we obtain the Lagrangian
    \begin{align}
        \int l\ d \mu
              -
              \frac1{2\tau}
              \mathrm{MMD}^2(\mu, \hat\mu)
              +\frac{\epsilon^2}{2\tau}
                            + \langle{f} , { \int \phi\ d\mu -h} \rangle_{\mathcal H}.
    \end{align}
    Since $\mathrm{MMD}(\mu,\hat{\mu}) = \|h - \hat{h}\|_{\mathcal H}$, we arrive at the Lagrange saddle-point problem 
  % By straightforward Lagrange duality and associating the linear constraint with the multiplier in the dual space, $f\in\rkhs$, i.e., considering the Lagrangian
  %   \begin{align}
  %       \langle 
  %             l, \mu
  %             \rangle
  %             -
  %             \frac1{2\tau}
  %             \mmd^2(\mu, \hat\mu)
  %             +\frac{\epsilon^2}{2\tau}
  %                           + \hiprod{f}{ \int \phi\dd\mu -h} .
  %   \end{align}
    % Since MMD is simply the RKHS norm, we arrive at the dual problem
  \begin{equation}
      \begin{aligned}
          \inf_{\tau>0, f\in\rkhs}\sup_{\mu \in \mathcal M,h\in\rkhs} \left\{ { \int l\ d \mu}
              - \frac1{2\tau}
              \hnorm{h - \hat{h}}^2
              +\frac{\epsilon^2}{2\tau}
              - \hiprod{f}{h} 
              +\int \hiprod{f}{\phi} \dd \mu
              \right\}
          \end{aligned}
    \end{equation}
     where $\hat{h}$ is the kernel mean embedding of the empirical measure $\hat{P}_n=\frac 1n\sum_{i=1}^n \delta_{x_i}$.
     Carrying out the quadratic optimization problem w.r.t. $h$ in closed form, 
     we eliminate the variable $h$.
     Finally, using the reproducing property, and rearranging the terms, we obtain the result in~\eqref{eq-kdro-dual-ap-1}. 
     An optimal choice of the dual variable $\tau$
  yields the equivalent reformulation~\eqref{eq-kdro-dual-ap-2}.
  Smoothing via the definition of the dual norm in the Hilbert spaces, i.e., using that $\|f\|_\rkhs = \max_{h \in \rkhs: \|h\|_{\rkhs}\leq 1} \la h,f \ra$, we obtain \eqref{eq:kdro_smooth-ap}.
  Note that the strong duality holds due to the infinite-dimensional linear program duality; see \cite{zhu_kernel_2021} for an elementary proof of strong duality.

  \paragraph*{Wasserstein setting.}
  The derivaiton is similar to the MMD setting by combining the dual Kantorovich representation of OT~\eqref{label-kantorovich-dual} and Lagrange duality. 
  \end{proof}

\subsection{Background of Gradient Flow and Geodesic Convexity}
% \subsection{Lemma~\eqref{thm:exist-rkhsgf}}
\label{sec:geod-cvx-w2}
Recall that a functional $\mathcal E$ defined on a Hilbert space $\rkhs$ is $\lambda$-convex if $\forall s \in[0,1], \forall u_0, u_1 \in \rkhs$,
\begin{align}
  \label{eq-convexity}
  \mathcal{E}((1-s) u_{0} + s u_1) \leq(1-s) \mathcal{E}(u_{0})+s \mathcal{E}(u_1)
%   \\
  -\frac{\lambda}{2} s(1-s) \hnorm{u_0 - u_1}^2.
\end{align}
If $\lambda>0$, $\mathcal{E}$ is strongly convex.
This notion of convexity does not make sense in the Wasserstein space.

  Recent machine learning literature has explored \emph{Wasserstein gradient flow} (WGF).
  The 2-Wasserstein space $(\mathcal M, W_2)$ is geodesically convex, as defined below.
%   , where we take
%   $\mnfdM= \probSimplex$, probability measures defined on the closure of a closed bounded convex domain $\Omega\subset \mathbb R^d$.
%   When there is ambiguity, we simply write $\mathcal P$ for conciseness.

In a metric space \((\mathcal M, \mathcal{D})\) a curve \(\gamma:[0,1] \rightarrow\)
\(\mathcal M\) is a (constant speed) geodesic if
$$
\forall r, s \in[0,1]: \quad \mathcal{D}(\gamma(r), \gamma(s))=|s-r| \mathcal{D}(\gamma(0), \gamma(1))
.
$$
We refer to that as the geodesic \(\gamma\) connects the points \(\gamma(0)\) and \(\gamma(1)\) and write
    Geod \((\gamma(0), \gamma(1))\) for the set of all such geodesics.
  \begin{definition}
    [Geodesic metric spaces]
    The metric space \((\mathcal M, \mathcal{D})\) is a geodesic space, if for all \(u_{0}, u_{1} \in \mathcal M\) there exists
    a geodesic connecting \(u_{0}\) and \(u_{1}\).
\end{definition}

\begin{definition}
    [Geodesic convexity]
    A functional \(\mathcal{E}: \mathcal M \rightarrow \mathbb{R}_{\infty}\) is geodesically \(\lambda\)-convex if
$
\forall u_{0}, u_{1} \in  \mathcal M \;\; \exists \gamma \in \operatorname{Geod}\left(u_{0}, u_{1}\right)
$ s.t.
$$
\mathcal{E}(\gamma(s)) \leq(1-s) \mathcal{E}(\gamma(0))+s \mathcal{E}(\gamma(1))-\frac{\lambda}{2} s(1-s) \mathcal{D}(\gamma(0), \gamma(1))^{2}, \  \forall s \in[0,1].
$$
\end{definition}
As we have seen in the main text, e.g., discussions around Example~\ref{ex:wdro-not-work}, this complication of convexity structure makes optimization in the Wasserstein space more difficult than in general Hilbert spaces, which motivates our approach to work in the RKHS.

Next, we mention standard results for RKHSGF in the context of this paper for completeness.
%;    we do not claim the novelty of such results.
% For readers unfamiliar with PDE gradient flows in the Hilbert space, we refer to \cite{ambrosio2008gradient} for complete treatment and \cite{santambrogio_euclidean_2017,mielke_gradients_nodate} for accessible introductions.
% In the discussion of RKHSGF below, we suppress the variable $\mu$ in the energy and write  $\mathcal E (f) :=  F (\mu, f)$, for some $\mu \in \mathcal M$.
% The purpose of using the minus sign is such that we can discuss the flow in RKHS as an energy minimization problem in using consistent terminologies as in the literature, as opposed to changing everything to the maximization setting.
First, we outline the existence, uniqueness, and a powerful result known as the \emph{evolutionary variational inequalities} \eviL.
Since an RKHS is a Hilbert space, Lemma~\ref{thm:exist-rkhsgf} is simply the \eviL in a Hilbert space, whose proof is standard~\citep{ambrosio2008gradient,santambrogio_optimal_2015,mielke_gradients_nodate}.

\begin{lemma}
    [Characterizations of RKHS gradient flow]
    \label{thm:exist-rkhsgf}
    Suppose the energy functional {$\mathcal V $} is proper, upper semicontinuous, $\lambda$-convex for some $\lambda\in\mathbb R$, and has compact sublevel sets.
    Then for any initial condition in the RKHS $f(0,x)\in \rkhs$, there exists a unique solution {to \eqref{eq-kgf}} at time $t$, $f(t)\in \rkhs$.
    
    Furthermore, the gradient flow solution $f(t,x)$ satisfies \pd{the following} \eviL, for $t,s\in [0, T]$: 
    \begin{align}
    \label{eq-evi}
     \frac{1}{2} \hnorm{f(t) - {h}}^2
     \leq 
     \frac{1}{2} \mathrm{e}^{-\lambda(t-s)}
                    \hnorm{f(s) - {h}}^2
    %  \\
     +M_\lambda(t-s)(
      {\mathcal V }({h}) - {\mathcal V }(f(t))
     )
     ,\quad \forall {h}\in \operatorname{dom}({\mathcal V })\subset\rkhs
  \end{align}
\end{lemma}
where $  M_\lambda(\tau) = \int _0^\tau e^{-\lambda (\tau -s )}\dd s$.
Using \eviL, we can effortlessly extract convergence results.
Suppose a minimizer  $f^* \in  \inf_{f\in \mathcal H} {\mathcal V }(f)$ of the energy exists, we set ${h} = f^*, s=0$ in \eviL and obtain
\begin{align}
   \hnorm{f(t) - f^*}^2
  \leq 
   \mathrm{e}^{-\lambda t}
                 \hnorm{f(0) - f^*}^2
  +2M_\lambda {t}
  \left(
    \inf_{f\in \rkhs}{\mathcal V }(f)
    -
    {\mathcal V }(f(t))
  \right)
  \leq 
   \mathrm{e}^{-\lambda t}
                 \hnorm{f(0) - f^*}^2
  ,
\end{align}
yielding an \emph{exponential convergence in time} if the energy ${\mathcal V }$ is strongly convex (in the usual sense) w.r.t. the $f$ variable,
i.e., $\lambda >0$.
Note that the convexity condition can be further weakened using functional inequalities such as the \L{}ojasiewicz inequality.

A time-discretization of the RKHS gradient flow~\eqref{eq-kgf} using the standard \emph{minimizing movement scheme} (MMS) in the gradient flow literature, is as follows
% , a.k.a the time-incremental minimization scheme,
\begin{align}
    \label{eq-mms-rkhs}
\hat{f}^{k+1}
    \in
        \arg\inf_{f\in\rkhs} \left\{ 
        {\mathcal V }(f)
            + \frac1{2\tau}
            \hnorm{f - \hat{f}^k}^2
            \right\}
            ,
            \\
            \hat{f}^0(x) = f(0,x)\in\rkhs
            ,
        \end{align}
where $\tau>0$ is the step size for time-discretization.
Defining the piecewise constant function
$\bar{f}_\tau(t,x) : = \hat{f}^{k}(x)$ for $t\in [k\tau, k\tau+\tau]$,
    standard PDE proofs (see, e.g., \cite{ambrosio2008gradient}) guarantee that $\bar{f}_\tau$ converges to the continuous-time RKHSGF solution, i.e.,
    $\bar{f}_\tau\to f(t)\textrm{ as }\tau\to 0$.
% Therefore, a natural by-product of the existence results for RKHSGF is the MMS step~\eqref{eq-mms-rkhs}. 
The fully-implicit Euler discretization is difficult to implement in practice.
In the next section, we will use the explicit step, coupled with measure-update step.

% By standard derivation, we have

% \begin{proposition}
%   [Wasserstein proximal step]
%   The optimal value of the optimization problem \eqref{eq:wgf-mu-update-one-step} is equivalent to
%   \begin{align}
%     % \inf_{\mu\in \mathcal M}
%     \E_\mu^l
%   \int
%   F'_\mu(\mu^k, f^k)(x)
%   +
%   \frac{1}{2\tau}\D(\mu, \mu^k)
%   + \frac{1}{2s }W_2^2(\mu, \mu^l).
%   \end{align}
% \end{proposition}

\section{THEORETICAL ANALYSIS OF PRIMAL-DUAL KERNEL MIRROR PROX IN SECTION \ref{S:MP}}

For the sake of generality, in particular, to cover the DRO problem \eqref{eq:kdro_smooth}, we consider a more general saddle-point problem. Theorems \ref{Th:ideal_MP_determ} and \ref{Th:ideal_MP_stoch} are obtained as corollaries of the results obtained in this section. To make this section self-contained and for the reader's convenience we repeat some definitions and results given in the main text.

We consider generic variable $x \in \mathbb{R}^n$ with domain $\mathcal{X}$. We denote by $ \mathcal M$  the set of all probability measures on $\mathcal{X}$ that admit densities w.r.t. the Lebesgue measure and the density is continuous and positive almost everywhere on $\mathcal{X}$.
%\footnote{\pd{We believe that this assumption may be weakened}} 
We also assume that there are two Hilbert spaces $\rkhs_f, \rkhs_h$, a convex set $\Theta \subseteq \mathbb{R}^d$, and convex a compact $H \in \rkhs_h$.
We consider the following general infinite-dimensional saddle-point problem
\begin{align}
\label{eq:SPP_general_app}
    \inf_{\theta \in \Theta \subseteq \mathbb{R}^d, f (x) \in \rkhs_f}
    \sup_{\mu\in \mathcal M, h (x) \in H \subseteq \rkhs_h}
    \quad F(\theta,f,\mu,h).
\end{align}
For shortness, we denote the set of all variables by $u=(\theta,f,\mu,h)$ and slightly abuse notation to define $F(u)=F(\theta,f,\mu,h)$.

Our first main assumption is as follows.
\begin{assumption}
\label{Asm:SPP_gen_convexity_app}
The functional $F(\theta,f,\mu,h)$ is convex in $(\theta,f)$ for fixed $(\mu,h)$ and concave in $(\mu,h)$ for fixed $(\theta,f)$. %Further, we assume that the objective $F$ is Fr\'echet differentiable w.r.t. each variable.
\end{assumption}

\subsection{Preliminaries}
To construct the mirror prox algorithm for problem \eqref{eq:SPP_general_app} we need to first introduce proximal setup, which consists of norms, their dual, and Bregman divergences on each space of the variables. 

For the space of the variable $\theta$, we introduce the standard proximal setup with the Euclidean norm $\|\cdot \|_2$,  distance-generating function $d_{\theta}(\theta) = \frac{1}{2}\|\theta\|_2^2$, which gives Bregman divergence $B_{\theta}(\theta,\Breve{\theta})=\frac{1}{2}\|\theta-\Breve{\theta}\|_2^2$. This leads to the mirror step defined as
\begin{equation}
    \theta_+ = \MD^{\theta,\Theta}_{\eta}(\theta_0, \xi_\theta) = \arg \min_{\tilde{\theta} \in \Theta} \{ \langle \tilde{\theta}, \eta\xi_\theta \rangle + \frac{1}{2}\|\tilde{\theta} - \theta_0\|_2^2\}.
\end{equation}
We note that our choice of the Euclidean proximal setup is made for simplicity and that other standard proximal setups as in  \cite{nemirovski_robust_2009} are possible.

For the space of the variable $f$, we use the norm of the Hilbert space $\|\cdot\|_{\rkhs_f}$, distance generating function $d_{f}(f) = \frac{1}{2}\|f\|_{\rkhs_f}^2$, which gives Bregman divergence $B_{\rkhs_f}(f,\Breve{f})=\frac{1}{2}\|f-\Breve{f}\|_{\rkhs_f}^2$. This leads to the mirror step
\begin{equation}
    f_+ = \MD^{f,{\rkhs_f}}_{\eta}(f_0,  \xi_f) = \arg \min_{\tilde{f}} \{ \langle \tilde{f}, \eta \xi_f \rangle + \frac{1}{2}\|\tilde{f} - f_0\|_{\rkhs_f}^2\} = f_0 - \eta\xi_f.
\end{equation}
Note that if $\xi_f =  \mathcal V '( f)$ (the Fr\'echet differential of the energy $\mathcal V$), the above equation is equivalent to $\frac{1}{\tau} (f_+-f) = - \mathcal V '( f)$, which can be seen as a time discretization of \eqref{eq-kgf-2-main}.

For the space of the variable $\mu$, we follow \citet{hsieh2019finding} and, first, introduce the Total Variation norm for the elements of $\mathcal{M}$
\[
\| \mu \|_{TV} = \sup_{\|\xi\|_{L^{\infty}}\leq 1} \int \xi d\mu = \sup_{\|\xi\|_{L^{\infty}}\leq 1} \la \xi ,\mu\ra,
\]
where $\|\xi\|_{L^{\infty}}$ is the $L^{\infty}$-norm of functions. 
To define the mirror step, we use (negative) Shannon entropy 
\begin{equation}
    \Phi(\mu) = \int d\mu \ln \frac{d\mu}{dx}
\end{equation}
and its Fenchel dual
\begin{equation}
    \Phi^*(\xi) =  \ln \int e^\xi dx
\end{equation}
defined for $\xi$ from the space $\mathcal{M}^*$ of all bounded integrable functions on $\mathcal{X}$. The corresponding Bregman divergence is the relative entropy given by  
\begin{equation}
    D_\Phi(\mu,\breve{\mu}) =  \int d\mu \ln \frac{d\mu}{ d\breve{\mu}}.
\end{equation}
This leads to the mirror step \citep[Theorem 1]{hsieh2019finding}
\begin{equation}
    \mu_+ = \MD^{\mu}_{\eta}(\mu_0,  \xi_\mu) = d \Phi^*(d \Phi(\mu_0) - \eta \xi_\mu) \quad  \equiv  \quad d\mu_+ = \frac{e^{-\eta \xi_\mu}d\mu_0}{\int e^{-\eta \xi_\mu}d\mu_0}.
\end{equation}
This step can be seen as a time discretization of the dynamics \eqref{eq-react-eq} as outlined in the discussion around \eqref{eq:md-intro}.

Finally, for the space of the variable $h$, we do the same as for the variable $f$. Namely, we use the distance generating function $d_{h}(h) = \frac{1}{2}\|h\|_{\rkhs_h}^2$, which gives Bregman divergence $B_{\rkhs_h}(h,\Breve{h})=\frac{1}{2}\|h-\Breve{h}\|_{\rkhs_h}^2$. This leads to the mirror step
\begin{equation}
    h_+ = \MD^{h,H}_{\eta}(h_0,  \xi_h) = \arg \min_{\tilde{h}\in H} \{ \langle \tilde{h}, \eta \xi_h \rangle + \frac{1}{2}\|\tilde{h} - h_0\|_{\rkhs_h}^2\} .
\end{equation}

Our second main assumption is as follows
\begin{assumption}
\label{Asm:SPP_gen_smooth_app}
The functional $F(\theta,f,\mu,h)$  is (Fr\'echet) differentiable in
$\theta$ w.r.t. the Euclidean norm, in $f$ w.r.t. the RKHS norm, in $\mu$ w.r.t.  $L^2$, and in $h$ w.r.t. the RKHS norm. Furthermore, we assume that the derivatives are Lipschitz continuous in the following sense
\begin{align}
&\|F'_\theta(u)-F'_\theta(\tilde{u})\|_2 \leq L_{\theta\theta}\|\theta-\tilde{\theta}\|_2 +  L_{\theta f}\|f-\tilde{f}\|_{\mathcal{H}_f} +  L_{\theta\mu}\|\mu-\tilde{\mu}\|_{TV} +   L_{\theta h}\|h-\tilde{h}\|_{\mathcal{H}_h}, \\
&\|F'_f(u)-F'_f(\tilde{u})\|_{\mathcal{H}_f} \leq L_{f\theta}\|\theta-\tilde{\theta}\|_2 +  L_{f f}\|f-\tilde{f}\|_{\mathcal{H}_f} +  L_{f\mu}\|\mu-\tilde{\mu}\|_{TV} +   L_{f h}\|h-\tilde{h}\|_{\mathcal{H}_h}, \\
&\|F'_\mu(u)-F'_\mu(\tilde{u})\|_{L^{\infty}} \leq L_{\mu\theta}\|\theta-\tilde{\theta}\|_2 +  L_{\mu f}\|f-\tilde{f}\|_{\mathcal{H}_f} +  L_{\mu\mu}\|\mu-\tilde{\mu}\|_{TV} +   L_{\mu h}\|h-\tilde{h}\|_{\mathcal{H}_h}, \\
&\|F'_h(u)-F'_h(\tilde{u})\|_{\mathcal{H}_h} \leq L_{h\theta}\|\theta-\tilde{\theta}\|_2 +  L_{h f}\|f-\tilde{f}\|_{\mathcal{H}_f} +  L_{h\mu}\|\mu-\tilde{\mu}\|_{TV} +   L_{h h}\|h-\tilde{h}\|_{\mathcal{H}_h}.
\end{align}
\end{assumption}
We also denote
\begin{equation}
\label{eq:L_general_app}
    L = \max_{\kappa_1,\kappa_2\in \{\theta,f,\mu,h\}}\{L_{\kappa_1\kappa_2} \}.
\end{equation}

\subsection{Mirror Prox Algorithm and Its Analysis}

The updates of the general infinite-dimensional Mirror Prox algorithm\footnote{We believe that the Dual Averaging algorithm \cite{nesterov2007dual} can be extended to our setting in a similar fashion as it was done in \cite{dvurechensky2015primal-dual} for saddle-point problems in Hilbert spaces.} for problem \eqref{eq:SPP_general_app} are given in Algorithm \ref{alg:mirror_prox_gen_app}.
\begin{algorithm}[ht!]
\caption{Ideal General Mirror-Prox}
{
\begin{algorithmic}[1]\label{alg:mirror_prox_gen_app}
    \REQUIRE{ Initial guess $\tilde{u}_0=(\tilde{\theta}_0,\tilde{f}_0,\tilde{\mu}_0,\tilde{h}_0)$, step-sizes $\eta_\theta,\eta_f,\eta_\mu,\eta_h >0$.}
    %such that $\ds\bt = [t_1^\top, \ldots, t_m^\top]^\top$, $\ds t_i = \argmin_{t\in\sT_i} d_{\sT_i}(t)$ for each $t_i\in\braces{x_i, p_i, y_i, q_i, z_i, s_i}$ and corresponding set $\sT_i\in\braces{\bar\sX, \sP_i, \bar\sY, \sQ_i, \R^{d_y}, \R^{d_x}}$, $i = 1, \ldots, m$.}
    \FOR{$k = 0, 1, \ldots, N - 1$}
        \STATE{Compute for $\tilde{u}_k=(\tilde{\theta}_k,\tilde{f}_k,\tilde{\mu}_k,\tilde{h}_k)$
            \begin{align*}
                & \theta_{k} = \MD^{\theta,\Theta}_{\eta_\theta}(\tilde{\theta}_k, F'_\theta(\tilde{u}_k)),  \qquad
                & f_{k} = \MD^{f,\mathcal{H}_f}_{\eta_f}(\tilde{f}_k,F'_f(\tilde{u}_k)),\\
                &\mu_{k} = \MD^{\mu}_{\eta_\mu}(\tilde{\mu}_k, -F'_\mu(\tilde{u}_k)), \qquad 
                & h_{k} = \MD^{h,H}_{\eta_h}(\tilde{h}_k, -F'_h(\tilde{u}_k)).
            \end{align*}  
            }
        \STATE{Compute for $ {u}_k=( {\theta}_k, {f}_k, {\mu}_k, {h}_k)$
            \begin{align*}
                & \tilde{\theta}_{k+1} = \MD^{\theta,\Theta}_{\eta_\theta}(\tilde{\theta}_k, F'_\theta(u_k)),  \qquad
                & \tilde{f}_{k+1} = \MD^{f,\mathcal{H}_f}_{\eta_f}(\tilde{f}_k,F'_f(u_k)),\\
                &\tilde{\mu}_{k+1} = \MD^{\mu}_{\eta_\mu}(\tilde{\mu}_k, -F'_\mu(u_k)), \qquad 
                & \tilde{h}_{k+1} = \MD^{h,H}_{\eta_h}(\tilde{h}_k, -F'_h(u_k)).
            \end{align*} 
            }
    \ENDFOR
    \STATE Compute $(\bar{\theta}_N,\bar{f}_N,\bar{\mu}_N,\bar{h}_N)=\bar{u}_N = \frac{1}{N}\sum_{k=0}^{N-1} u_k = \frac{1}{N}\sum_{k=0}^{N-1}( {\theta}_k, {f}_k, {\mu}_k, {h}_k)$.
    % \ENSURE{For $\bt\in\braces{\bx, \bp, \by, \bq, \bs, \bz}$ compute $\ds\hat \bt^{N} = \frac{1}{N}\sum_{k=0}^{N-1} \bt^{k+\frac{1}{2}}$.}
\end{algorithmic}
}
\end{algorithm}

For the analysis of the mirror prox algorithm, we need the following auxiliary results. The first one is used for the mirror steps applied to the variables $\theta,f,h$.

\begin{lemma}
\label{Lm:rkhs_step_app}
Let $\mathcal H$ be (possibly finite-dimensoinal) Hilbert space and let $H\subset \mathcal H$ be convex and closed. Let $\tilde{h} \in H$ and $\xi, \tilde{\xi} \in \mathcal H^*=\mathcal H$, $\eta >0$, and 
\begin{align}
    h&= \arg \min_{\hat{h}\in H} \left\{\la \hat{h}, \eta \xi   \ra + \frac{1}{2}\|\tilde{h}-\hat{h}\|_{\mathcal{H}}^2\right\} = \MD^{h,H}_{\eta}(\tilde{h},  \xi), \label{eq:Lm:rkhs_step_1_app}\\ 
    \tilde{h}_+&= \arg \min_{\hat{h}\in H} \left\{\la \hat{h}, \eta \tilde{\xi}  \ra + \frac{1}{2}\|\tilde{h}-\hat{h}\|_{\mathcal{H}}^2\right\} = \MD^{h,H}_{\eta}(\tilde{h},  \tilde{\xi}). \label{eq:Lm:rkhs_step_2_app}
\end{align}
Then, for any $\hat{h} \in H$
\begin{equation}
    \la h - \hat{h}, \eta \tilde{\xi} \ra \leq 
    \frac{1}{2}\|\hat{h}-\tilde{h}\|_{\mathcal{H}}^2 
    - \frac{1}{2}\|\hat{h}-\tilde{h}_+\|_{\mathcal{H}}^2 
    + \frac{\eta^2}{2}\|\tilde{\xi}-\xi\|_{\mathcal{H}}^2 
    - \frac{1}{2}\|h-\tilde{h}\|_{\mathcal{H}}^2.
\end{equation}
\end{lemma}
\begin{proof}
By the optimality condition in \eqref{eq:Lm:rkhs_step_2_app}, we have for all $\hat{h} \in H$
\begin{equation}
    \la \eta \tilde{\xi}-(\tilde{h}-\tilde{h}_+), \hat{h}-\tilde{h}_+\ra \geq 0.
\end{equation}
Rearranging, we obtain
\begin{equation}
\label{eq:Lm:rkhs_step_proof_1}
\la \tilde{h}_+ - \hat{h}, \eta \tilde{\xi} \ra \leq \la \tilde{h}_+ - \hat{h}, \tilde{h}-\tilde{h}_+ \ra = -\frac{1}{2} \|\tilde{h}_+-\tilde{h}\|_\rkhs^2-\frac{1}{2}\|\hat{h}-\tilde{h}_+\|_\rkhs^2+\frac{1}{2}\|\hat{h}-\tilde{h}\|_\rkhs^2.
\end{equation}
In the same way, by the optimality condition in \eqref{eq:Lm:rkhs_step_1_app}, we have for all $\hat{h} \in H$, and, in particular for $\tilde{h}_+$
\begin{equation}
    \la \eta \xi-(\tilde{h}-h), \tilde{h}_+-h\ra \geq 0.
\end{equation}
Rearranging, we obtain
\begin{equation}
\la h - \tilde{h}_+, \eta  \xi \ra \leq \la h-\tilde{h}_+, \tilde{h}-h \ra = -\frac{1}{2} \|h-\tilde{h}\|_\rkhs^2-\frac{1}{2}\|\tilde{h}_+-h\|_\rkhs^2+\frac{1}{2}\|\tilde{h}_+-\tilde{h}\|_\rkhs^2.
\end{equation}
Combining the last inequality with \eqref{eq:Lm:rkhs_step_proof_1} and using the Fenchel inequality, we obtain
\begin{align}
 \la h - \hat{h}, \eta  \tilde{\xi} \ra &=  \la \tilde{h}_+ - \hat{h}, \eta \tilde{\xi} \ra + \la h - \tilde{h}_+, \eta  \xi \ra + \la h - \tilde{h}_+, \eta  (\tilde{\xi}-\xi) \ra \\
& \leq -\frac{1}{2} \|\tilde{h}_+-\tilde{h}\|_\rkhs^2-\frac{1}{2}\|\hat{h}-\tilde{h}_+\|_\rkhs^2+\frac{1}{2}\|\hat{h}-\tilde{h}\|_\rkhs^2 \\
& -\frac{1}{2} \|h-\tilde{h}\|_\rkhs^2-\frac{1}{2}\|\tilde{h}_+-h\|_\rkhs^2+\frac{1}{2}\|\tilde{h}_+-\tilde{h}\|_\rkhs^2 \\
& + \frac{\eta^2}{2}\|\tilde{\xi}-\xi\|_{\mathcal{H}}^2 
    + \frac{1}{2}\|h-\tilde{h}\|_{\mathcal{H}}^2,
\end{align}
which gives the result of the Lemma.
\end{proof}

The second result characterizes the mirror step with respect to the measure $\mu$. 

\begin{lemma}[{\cite{hsieh2019finding} [Lemma 5]}]
\label{Lm:measure_step_app}
Let $\tilde{\mu} \in \mathcal M$ and $\xi, \tilde{\xi} \in  \mathcal{M}^*$, $\eta >0$, and 
\begin{align}
    \mu&= \MD^{\mu}_{\eta}(\tilde{\mu},  \xi), \\ 
    \tilde{\mu}_+&= \MD^{\mu}_{\eta}(\tilde{\mu},  \tilde{\xi}).
\end{align}
Then, for any $\hat{\mu} \in \mathcal M$
\begin{equation}
    \la \mu - \hat{\mu}, \eta \tilde{\xi} \ra \leq 
    D_{\Phi}(\hat{\mu},\tilde{\mu}) 
    - D_{\Phi}(\hat{\mu},\tilde{\mu}_+)
    + \frac{\eta^2}{8}\|\tilde{\xi}-\xi\|_{L^{\infty}}^2 
    - 2\|\mu-\tilde{\mu}\|_{TV}^2.
\end{equation}
\end{lemma}

The following result gives the convergence rate of Algorithm \ref{alg:mirror_prox_gen_app}.
\begin{theorem} 
\label{Th:ideal_MP_determ_app}
Let Assumptions \ref{Asm:SPP_gen_convexity_app}, \ref{Asm:SPP_gen_smooth_app} hold. Let also the step sizes in Algorithm \ref{alg:mirror_prox_gen_app} satisfy $\eta_\theta=\eta_f=\eta_\mu=\eta_h = \frac{1}{16L}$, where $L$ is defined in \eqref{eq:L_general_app}. Then, for any compact set $U = U_\theta \times U_f \times U_\mu \times U_h \subseteq \Theta \times \mathcal{H}_f \times \mathcal{M} \times H$, the sequence $(\bar{\theta}_N,\bar{f}_N,\bar{\mu}_N,\bar{h}_N)$ generated by Algorithm \ref{alg:mirror_prox_gen_app} satisfies
\begin{align*}
&\max_{\mu \in U_\mu, h \in U_h} F(\bar{\theta}_N,\bar{f}_N,\mu,h) - \min_{\theta \in U_\theta, f \in U_f}F(\theta,f,\bar{\mu}_N,\bar{h}_N) \\
    & \leq \frac{8L}{N}\max_{u \in U} \left(\|\theta-\tilde{\theta}_0\|_2^2 + \|f-\tilde{f}_0\|_{\mathcal{H}_f}^2 + 2 D_{\Phi}(\mu,\tilde{\mu}_0)  +  \|h-\tilde{h}_0\|_{\mathcal{H}_h}^2 \right).
\end{align*} 
\end{theorem}
Before we prove the theorem, we would like to underline that the l.h.s. of the estimate is an appropriate measure of the suboptimality. 
This notion of the duality gap is quite standard for algorithms for saddle-point problems, see, e.g., Sect 5.3.6.1 of \cite{ben-tal2023lectures}.
 The motivation is that the saddle-point problem can be considered as a primal-dual pair of optimization problems
$Opt(P)=\min_{\theta,f} \{ \bar F (\theta,f):= \max_{\mu \in U_\mu, h \in U_h} F(\theta,f,\mu,h)\}$, $Opt(D)=\max_{\mu,h} \{ \underline{F} (\mu,h):= \min_{\theta \in U_\theta, f \in U_f} F(\theta,f,\mu,h)\}$. By the duality we have that
\begin{align*}
    &\bar F (\bar{\theta}_N, \bar{f}_N) - Opt(P) + (Opt(D) - \underline{F} (\bar{\mu}_N,\bar{h}_N)) \\
    &= \max_{\mu \in U_\mu, h \in U_h  } F(\bar{\theta}_N, \bar{f}_N,\mu,h) - Opt(P) + Opt(D) - \min_{\theta \in U_\theta, f \in U_f}F( \theta, f,\bar{\mu}_N, \bar{h}_N ) \\
    & = \max_{\mu \in U_\mu, h \in U_h} F(\bar{\theta}_N,\bar{f}_N,\mu,h) - \min_{\theta \in U_\theta, f \in U_f}F(\theta,f,\bar{\mu}_N,\bar{h}_N).
\end{align*}
Thus, our convergence rate implies simultaneously the convergence rate in the primal and in the dual problem.

We also underline that the algorithm does not need the set $U = U_\theta \times U_f \times U_\mu \times U_h $ as input and that it solves the original problem on the set $\Theta \times \mathcal{H}_f \times \mathcal{M} \times H$. The set $U$ appears only in the convergence result and this can be any compact set. In particular, if $U$ contains a neighborhood of the solution to the original problem, then the original problem is equivalent to the same problem but on the set $U$. Additionally, saddle-point problems can be equivalently reformulated as variational inequalities, see Example 2.2 in \cite{antonakopoulos2019adaptive}. Then, our notion of duality gap in the l.h.s. of the bound in the above Theorem can be bounded by the restricted gap function defined in (2.5a) of \cite{antonakopoulos2019adaptive}, see also Sect. 6 of \cite{stonyakin2022generalized}. Then, Lemma 1 in \cite{antonakopoulos2019adaptive} implies that $U$ can be an arbitrary compact set that contains a neighborhood of the solution.

\begin{proof}
Applying Lemma \ref{Lm:rkhs_step_app} to the step  in $\theta $, we obtain for any $\theta \in \Theta$ and for $k=0,...,N-1$
\begin{align}
\label{eq:Th:ideal_MP_determ_proof_1}
    & \la \theta_{k} - \theta , \eta_\theta F'_\theta(u_k) \ra \leq \frac{1}{2}\|\theta-\tilde{\theta}_k\|_2^2 - \frac{1}{2}\|\theta-\tilde{\theta}_{k+1}\|_2^2- \frac{1}{2}\|\theta_k-\tilde{\theta}_k\|_2^2 + \frac{\eta_\theta^2}{2}\|F'_\theta(u_k)-F'_\theta(\tilde{u}_k)\|_2^2.
\end{align}  
Summing these inequalities for $k=0,...,N-1$, we obtain
\begin{align*}
    & \sum_{k=0}^{N-1}\la \theta_{k} - \theta ,  F'_\theta(u_k) \ra \leq \frac{1}{2\eta_\theta}\|\theta-\tilde{\theta}_0\|_2^2 + \frac{R_{\theta}}{2\eta_\theta}, \\
    & R_{\theta}=   \sum_{k=0}^{N-1} \left(-  \|\theta_k-\tilde{\theta}_k\|_2^2 + \eta_\theta^2\|F'_\theta(u_k)-F'_\theta(\tilde{u}_k)\|_2^2 \right).
\end{align*}  
In the same way, we obtain for all $f \in \mathcal{H}_f$
\begin{align*}
    & \sum_{k=0}^{N-1}\la f_{k} - f ,  F'_f(u_k) \ra \leq \frac{1}{2\eta_f}\|f-\tilde{f}_0\|_{\mathcal{H}_f}^2 + \frac{R_{f}}{2\eta_f}, \\
    & R_{f}=   \sum_{k=0}^{N-1} \left(-  \|f_k-\tilde{f}_k\|_{\mathcal{H}_f}^2 + \eta_f^2\|F'_f(u_k)-F'_f(\tilde{u}_k)\|_{\mathcal{H}_f}^2 \right)
\end{align*} 
and for all $h \in H$
\begin{align*}
    & \sum_{k=0}^{N-1}\la h_{k} - h ,  -F'_h(u_k) \ra \leq \frac{1}{2\eta_h}\|h-\tilde{h}_0\|_{\mathcal{H}_h}^2 + \frac{R_{h}}{2\eta_h}, \\
    & R_{h}=   \sum_{k=0}^{N-1} \left(-  \|h_k-\tilde{h}_k\|_{\mathcal{H}_h}^2 + \eta_h^2\|F'_h(u_k)-F'_h(\tilde{u}_k)\|_{\mathcal{H}_h}^2 \right).
\end{align*} 
Finally, applying Lemma  \ref{Lm:measure_step_app} to the step in $\mu$, we obtain for any $ \mu \in \mathcal{M}$
\begin{align*}
    & \sum_{k=0}^{N-1}\la \mu_{k} - \mu ,  -F'_\mu(u_k) \ra \leq \frac{1}{\eta_\mu} D_{\Phi}(\mu,\tilde{\mu}_0) + \frac{R_{\mu}}{2\eta_\mu}, \\
    & R_{\mu}=   \sum_{k=0}^{N-1} \left(-  \|\mu_k-\tilde{\mu}_k\|_{TV}^2 + \eta_\mu^2\|F'_\mu(u_k)-F'_\mu(\tilde{u}_k)\|_{L^{\infty}}^2 \right).
\end{align*} 

By convexity of $F$ in $(\theta,f)$ and concavity of $F$ in $(\mu,h)$, we have, for all $(\theta,f) \in \Theta \times \mathcal{H}_f$
\begin{align*}
    & \frac{1}{N} \sum_{k=0}^{N-1} F(u_k) - F(\theta,f,\bar{\mu}_N,\bar{h}_N) \leq \frac{1}{N} \sum_{k=0}^{N-1} (F(u_k) - F(\theta,f,\mu_k,h_k)) \\
    & \leq \frac{1}{N}\sum_{k=0}^{N-1}(\la \theta_{k} - \theta ,  F'_\theta(u_k) \ra + \la f_{k} - f ,  F'_f(u_k) \ra ) \\
    & \leq \frac{1}{2N\eta_\theta}\|\theta-\tilde{\theta}_0\|_2^2 + \frac{R_{\theta}}{2N\eta_\theta} +  \frac{1}{2N\eta_f}\|f-\tilde{f}_0\|_{\mathcal{H}_f}^2 + \frac{R_{f}}{2N\eta_f}.
\end{align*} 
In the same way, we obtain that, for all $(\mu,h) \in \mathcal{M} \times H$
\begin{align*}
    & -\frac{1}{N} \sum_{k=0}^{N-1} F(u_k) + F(\bar{\theta}_N,\bar{f}_N,\mu,h) \leq \frac{1}{N} \sum_{k=0}^{N-1} (-F(u_k) + F(\theta_k,f_k,\mu,h)) \\
    & \leq \frac{1}{N}\sum_{k=0}^{N-1}(\la \mu_{k} - \mu ,  - F'_\mu(u_k) \ra + \la h_{k} - h , - F'_h(u_k) \ra ) \\
    & \leq \frac{1}{N\eta_\mu}D_{\Phi}(\mu,\tilde{\mu}_0) + \frac{R_{\mu}}{2N\eta_\mu} +  \frac{1}{2N\eta_h}\|h-\tilde{h}_0\|_{\mathcal{H}_h}^2 + \frac{R_{h}}{2N\eta_h}.
\end{align*} 
Combining the last two bounds, we obtain that for all $\theta \in \Theta, f \in \mathcal{H}_f, \mu \in \mathcal{M}, h \in H$ it holds that
\begin{align*}
    & F(\bar{\theta}_N,\bar{f}_N,\mu,h) - F(\theta,f,\bar{\mu}_N,\bar{h}_N) \\
    & \leq \frac{1}{2N\eta_\theta}\|\theta-\tilde{\theta}_0\|_2^2 +  \frac{1}{2N\eta_f}\|f-\tilde{f}_0\|_{\mathcal{H}_f}^2 + \frac{1}{N\eta_\mu}D_{\Phi}(\mu,\tilde{\mu}_0)  +  \frac{1}{2N\eta_h}\|h-\tilde{h}_0\|_{\mathcal{H}_h}^2 \\
    & + \frac{1}{2N}\left(\frac{R_{\theta}}{\eta_\theta}+ \frac{R_{f}}{\eta_f}+ \frac{R_{\mu}}{\eta_\mu} + \frac{R_{h}}{\eta_h}\right).
\end{align*} 

Our next goal is to show that
\[
\frac{R_{\theta}}{\eta_\theta}+ \frac{R_{f}}{\eta_f}+ \frac{R_{\mu}}{\eta_\mu} + \frac{R_{h}}{\eta_h} \leq 0.
\]
Using the Lipschitz condition in Assumption \ref{Asm:SPP_gen_smooth_app}, we obtain
\begin{align*}
R_{\theta} &=   \sum_{k=0}^{N-1} \left(-  \|\theta_k-\tilde{\theta}_k\|_2^2 + \eta_\theta^2\|F'_\theta(u_k)-F'_\theta(\tilde{u}_k)\|_2^2 \right)\\
& \leq   \sum_{k=0}^{N-1} \left(-  \|\theta_k-\tilde{\theta}_k\|_2^2 + 4\eta_\theta^2 (L_{\theta \theta}^2 \|\theta_k-\tilde{\theta}_k\|_2^2 + L_{\theta f}^2  \|f_k-\tilde{f}_k\|_{\mathcal{H}_f}^2 +L_{\theta \mu}^2 \|\mu_k-\tilde{\mu}_k\|_{TV}^2 +L_{\theta h}^2 \|h_k-\tilde{h}_k\|_{\mathcal{H}_h}^2) \right).
\end{align*} 
Combining this with the similar estimates for $R_{f},R_{\mu},R_{h}$ and rearranging the terms, we obtain
\begin{align*}
\frac{R_{\theta}}{\eta_\theta}+ \frac{R_{f}}{\eta_f}+ \frac{R_{\mu}}{\eta_\mu} + \frac{R_{h}}{\eta_h}   \leq  \sum_{k=0}^{N-1} \Big(& \|\theta_k-\tilde{\theta}_k\|_2^2 (-1/\eta_\theta + 4(L_{\theta\theta}^2\eta_\theta + L_{f\theta}^2 \eta_f + L_{\mu \theta}^2 \eta_{\mu} + L_{h \theta}^2 \eta_h)) \\ 
&+ \|f_k-\tilde{f}_k\|_{\mathcal{H}_f}^2 (-1/\eta_f + 4(L_{\theta f}^2\eta_\theta + L_{ff}^2 \eta_f + L_{\mu f}^2 \eta_{\mu} + L_{h f}^2 \eta_h)) \\
&+ \|\mu_k-\tilde{\mu}_k\|_{TV}^2 (-1/\eta_\mu + 4(L_{\theta \mu}^2\eta_\theta + L_{f\mu}^2 \eta_f + L_{\mu \mu}^2 \eta_{\mu} + L_{h \mu}^2 \eta_h))\\
&+ \|h_k-\tilde{h}_k\|_{\mathcal{H}_h}^2 (-1/\eta_h+ 4(L_{\theta h}^2\eta_\theta + L_{fh}^2 \eta_f + L_{\mu h}^2 \eta_{\mu} + L_{h h}^2 \eta_h)) \Big) \leq 0,
\end{align*} 
where we used that $\eta_\theta,\eta_f,\eta_\mu,\eta_h \leq \frac{1}{16L}$ for $L$  defined in \eqref{eq:L_general_app}.

Thus, we finally obtain that for any compact $U = U_\theta \times U_f \times U_\mu \times U_h \subset \Theta \times \mathcal{H}_f \times \mathcal{M} \times H$
\begin{align*}
&\max_{\mu \in U_\mu, h \in U_h} F(\bar{\theta}_N,\bar{f}_N,\mu,h) - \min_{\theta \in U_\theta, f \in U_f}F(\theta,f,\bar{\mu}_N,\bar{h}_N) \\
    & \leq \frac{8L}{N}\max_{u \in U} \left(\|\theta-\tilde{\theta}_0\|_2^2 + \|f-\tilde{f}_0\|_{\mathcal{H}_f}^2 + 2 D_{\Phi}(\mu,\tilde{\mu}_0)  +  \|h-\tilde{h}_0\|_{\mathcal{H}_h}^2 \right).
\end{align*}

    % $\theta,f,h$ and Lemma  \ref{Lm:rkhs_step} to the step in $\mu$, we obtain for any $\theta \in \Theta, f \in \mathcal{H}_f, \mu \in \mathcal{M}, h \in H$ and for $k=0,...,N-1$
\end{proof}

\subsection{Analysis in the Stochastic Case}

To account for potential inexactness in the first-order information, we assume that instead of exact derivatives, the algorithm uses their inexact counterparts $\tilde{F}'_\theta(u),\tilde{F}'_f(u),\tilde{F}'_\mu(u),\tilde{F}'_h(u),$ that may be random and are assumed to satisfy the following assumption.
\begin{assumption}
\label{Asm:SPP_gen_stoch_app}
\begin{align}
&F'_\theta(u)=\E\tilde{F}'_\theta(u), \\
&F'_f(u)=\E\tilde{F}'_f(u), \\
&F'_\mu(u)=\E\tilde{F}'_\mu(u), \\
&F'_h(u)=\E\tilde{F}'_h(u),
\end{align}
\begin{align}
&\E\|F'_\theta(u)-\tilde{F}'_\theta(u)\|_2^2 \leq  \sigma_\theta^2, \\
&\E\|F'_f(u)-\tilde{F}'_f(u)\|_{\mathcal{H}_f}^2 \leq \sigma^2_f, \\
&\E\|F'_\mu(u)-\tilde{F}'_\mu(u)\|_{L^{\infty}}^2 \leq \sigma^2_\mu, \\
&\E\|F'_h(u)-\tilde{F}'_h(u)\|_{\mathcal{H}_h}^2 \leq \sigma^2_h.
\end{align}
\end{assumption}

\begin{theorem}
\label{Th:ideal_MP_stoch_app}
Let Assumptions \ref{Asm:SPP_gen_convexity_app}, \ref{Asm:SPP_gen_smooth_app}, and \ref{Asm:SPP_gen_stoch_app} hold. Let also in Algorithm \ref{alg:mirror_prox_gen_app} the stochastic derivatives be used instead of the deterministic and the step sizes satisfy $\eta_\theta=\eta_f=\eta_\mu=\eta_h=\frac{1}{16L}$, where $L$ is defined in \eqref{eq:L_general_app}. Then, for any compact set $U = U_\theta \times U_f \times U_\mu \times U_h \subseteq \Theta \times \mathcal{H}_f \times \mathcal{M} \times H$, the sequence $(\bar{\theta}_N,\bar{f}_N,\bar{\mu}_N,\bar{h}_N)$ generated by Algorithm \ref{alg:mirror_prox_gen_app} satisfies 
\begin{align*}
&\E \left\{\max_{\mu \in U_\mu, h \in U_h}  F(\bar{\theta}_N,\bar{f}_N,\mu,h) - \min_{\theta \in U_\theta, f \in U_f}  F(\theta,f,\bar{\mu}_N,\bar{h}_N) \right\} \\
    & \leq \frac{8L}{N}\max_{u \in U} \left(\|\theta-\tilde{\theta}_0\|_2^2 + \|f-\tilde{f}_0\|_{\mathcal{H}_f}^2 + 2D_{\Phi}(\mu,\tilde{\mu}_0)  +  \|h-\tilde{h}_0\|_{\mathcal{H}_h}^2 \right) + \frac{3}{16L} (\sigma^2_\theta+\sigma^2_f+\sigma^2_\mu+\sigma^2_h).
\end{align*} 
\end{theorem}
\begin{proof}
We proceed as in the proof of Theorem \ref{Th:ideal_MP_determ_app} changing in Algorithm \ref{alg:mirror_prox_gen_app} the exact first-order information to its inexact counterpart. In this way, we obtain the following counterpart of \eqref{eq:Th:ideal_MP_determ_proof_1} 
\begin{align}
    & \la \theta_{k} - \theta , \eta_\theta \tilde{F}'_\theta(u_k) \ra \leq \frac{1}{2}\|\theta-\tilde{\theta}_k\|_2^2 - \frac{1}{2}\|\theta-\tilde{\theta}_{k+1}\|_2^2- \frac{1}{2}\|\theta_k-\tilde{\theta}_k\|_2^2 + \frac{\eta_\theta^2}{2}\|\tilde{F}'_\theta(u_k)-\tilde{F}'_\theta(\tilde{u}_k)\|_2^2.
\end{align}  
Using the inequality
\begin{align}
     \E\|\tilde{F}'_\theta(u_k)-\tilde{F}'_\theta(\tilde{u}_k)\|_2^2 &\leq 3 \E\left(
    \|\tilde{F}'_\theta(u_k)-F'_\theta(u_k)\|_2^2 
    +\|\tilde{F}'_\theta(\tilde{u}_k)-F'_\theta(\tilde{u}_k)\|_2^2 +\|F'_\theta(u_k)-F'_\theta(\tilde{u}_k)\|_2^2  \right) \\
    & \stackrel{\text{Assumpt.} \ref{Asm:SPP_gen_stoch_app}}{\leq} 6 \sigma_\theta^2 + 3 \E \|F'_\theta(u_k)-F'_\theta(\tilde{u}_k)\|_2^2
\end{align}  
and taking the expectation in the previous inequality, we obtain the following counterpart of \eqref{eq:Th:ideal_MP_determ_proof_1} 
\begin{align}
     \la \theta_{k} - \theta , \eta_\theta F'_\theta(u_k) \ra \leq &\frac{1}{2}\E\|\theta-\tilde{\theta}_k\|_2^2 - \frac{1}{2}\E\|\theta-\tilde{\theta}_{k+1}\|_2^2- \frac{1}{2}\E\|\theta_k-\tilde{\theta}_k\|_2^2 \\
    & + \frac{3\eta_\theta^2}{2}\E\|F'_\theta(u_k)-F'_\theta(\tilde{u}_k)\|_2^2 + 3 \eta_\theta^2 \sigma_\theta^2.
\end{align}  

Repeating the same steps as in the proof of Theorem \ref{Th:ideal_MP_determ_app}, we obtain that 
for any compact $U = U_\theta \times U_f \times U_\mu \times U_h \subset \Theta \times \mathcal{H}_f \times \mathcal{M} \times H$
\begin{align}
& \E \left\{\max_{\mu \in U_\mu, h \in U_h}  F(\bar{\theta}_N,\bar{f}_N,\mu,h) - \min_{\theta \in U_\theta, f \in U_f}  F(\theta,f,\bar{\mu}_N,\bar{h}_N) \right\} \notag\\
    & \leq \frac{8L}{N}\max_{u \in U} \left(\|\theta-\tilde{\theta}_0\|_2^2 + \|f-\tilde{f}_0\|_{\mathcal{H}_f}^2 + 2D_{\Phi}(\mu,\tilde{\mu}_0)  +  \|h-\tilde{h}_0\|_{\mathcal{H}_h}^2 \right) + \frac{3}{16L} (\sigma^2_\theta+\sigma^2_f+\sigma^2_\mu+\sigma^2_h). \label{eq:Th:ideal_MP_stoch_proof_4}
\end{align} 
\end{proof}
Let us denote $\sigma^2=\sigma^2_\theta+\sigma^2_f+\sigma^2_\mu+\sigma^2_h$. 
As we see, Theorem \ref{Th:ideal_MP_stoch_app} guarantees the same convergence rate as in the exact case, but up to some vicinity which is governed by the level of noise. In most cases, the $\sigma^2/L$ term can be made of the same order $1/N$ by using mini-batching technique. Indeed, a mini-batch of size $N$ allows to change the variance from $\sigma^2$ to $\sigma^2/N$. Yet, we note that in this case, $N$ iterations will require the number of samples $O(N^2)$.

An alternative would be to use the information about the diameter of the set $U$. Indeed, assume that 
\[
\max_{u \in U} \left(\|\theta-\tilde{\theta}_0\|_2^2 + \|f-\tilde{f}_0\|_{\mathcal{H}_f}^2 + 2D_{\Phi}(\mu,\tilde{\mu}_0)  +  \|h-\tilde{h}_0\|_{\mathcal{H}_h}^2 \right) \leq \Omega_{U}^2.
\]
Then, we obtain the following counterpart of the r.h.s. of \eqref{eq:Th:ideal_MP_stoch_proof_4} substituting $\eta_\theta=\eta_f=\eta_\mu=\eta_h=\eta$
\[
\frac{\Omega_{U}^2}{2N \eta} + 3\sigma^2\eta.
\]
Fixing the number of steps $N$ and choosing
\[
\eta_\theta=\eta_f=\eta_\mu=\eta_h =\eta=\min\left\{ \frac{1}{16L}, \frac{\Omega_{U} \sigma }{ \sqrt{6N}} \right\},
\]
we obtain the following result
\begin{align}
& \E \left\{\max_{\mu \in U_\mu, h \in U_h}  F(\bar{\theta}_N,\bar{f}_N,\mu,h) - \min_{\theta \in U_\theta, f \in U_f} F(\theta,f,\bar{\mu}_N,\bar{h}_N) \right\} \notag\\
    & \leq \max \left \{\frac{8L\Omega_{U}^2}{N}, \sqrt{\frac{3\sigma^2 \Omega_{U}^2}{2N}}\right\}.
\end{align}

\section{DERIVATIONS OF RESULTS FOR DISTRIBUTIONALLY ROBUST OPTIMIZATION}
\subsection{Set-up and Adaptation of the General KMP Algorithm}
In this subsection we particularize the elements of Algorithm \ref{alg:mirror_prox_gen_app} for the specific DRO problem \eqref{eq:kdro_smooth}. 
%Compared to problem \eqref{eq:SPP_general} it has two additional variables $\theta \in \mathbb{R}^d$ and $h \in H \subset \rkhs$. For these variables, the proximal setup is introduced in the same way as for the variable $f$. 
We choose $\rkhs_f=\rkhs_h=\rkhs$ to be a reproducing kernel Hilbert space with kernel $k$.

Our main assumptions for this problem are 
\begin{itemize}
    \item $l$ is convex w.r.t. $\theta$.
    \item $L_0=\sup_{x,\theta}\|\nabla_{\theta} l(\theta;x)\|_2 < +\infty$. 
    \item  $\nabla_{\theta} l(\theta;x)$ is $L(x)$-Lipschitz w.r.t. $\theta$ and $L_1=\sup_{\mu}\E_{x \sim \mu} L(x)^2 < +\infty$.
    \item $C=\sup _{x} k(x, x)<+\infty$. 
\end{itemize}
Clearly, then the objective $F$ is convex in $(\theta,f)$ for fixed $(\mu,h)$ and concave in $(\mu,h)$ for fixed $(\theta,f)$. 

The Frechet derivatives of $F$ with respect to the variables $(\theta,f,\mu,h)$ are given by
\begin{align}
     F'_\theta &= \E_{x \sim \mu} \nabla_{\theta} l(\theta;x) \\
     F'_f &= \int k(x,x') d \hat{\mu}(x') + \epsilon h (x) - \int k(x,x') d \mu (x')  = \E_{x\sim\hat{\mu}} k(\cdot,x) + \epsilon h (\cdot)  -\E_{x\sim \mu} k(\cdot,x) \\
     - F'_\mu &= f(\cdot) - l(\theta;\cdot) \\
     - F'_h &= - \epsilon f(\cdot).
\end{align}

Since the derivatives w.r.t. $\theta$ and $f$ have the form of expectation, we can use the following stochastic counterparts.
We can take a sample of $X_i$'s from $\mu$ to construct an unbiased stochastic derivative
\begin{equation}
\label{eq:DRO_F_theta_stoch}
     \tilde{F}'_\theta  = \frac{1}{N_\theta} \sum_{i=1}^{N_\theta} \nabla_{\theta} l(\theta;X_i).
\end{equation}
Similarly, we can take a sample of $X_i$'s from $\mu$ and $\hat{X}_i$ from $\hat{\mu}$  to construct an unbiased stochastic derivative
\begin{equation}
\label{eq:DRO_F_f_stoch}
    \tilde{F}'_f = \epsilon h (\cdot) + \frac{1}{N_f} \sum_{i=1}^{N_f} (k(\cdot,\hat{X}_i) + k(\cdot,X_i)).
\end{equation}

The Lipschitz constants of the derivatives are estimated in the following way.
The derivative $F'_\theta$ depends only on $\mu$ and $\theta$. Thus, $L_{\theta f}=L_{\theta h}=0$. Further, we have
\begin{align}
&\left\| 
\E_{x \sim \mu} \nabla_{\theta} l(\theta_1;x) - \E_{x \sim \mu} \nabla_{\theta} l(\theta_2;x) \right\|_2  \leq  \E_{x \sim \mu} L(x)   \| \theta_1 - \theta_2 \|_2 
\end{align}
and $L_{\theta \theta} = L_1$.
\begin{align}
&\left\| 
\E_{x \sim \mu_1} \nabla_{\theta} l(\theta;x) - \E_{x \sim \mu_2} \nabla_{\theta} l(\theta;x) \right\|_2 \leq  L_0  \|\mu_1 - \mu_2\|_{TV}
\end{align}
and $L_{\theta \mu} = L_0$.

The derivative $F'_f$ depends only on $\mu$ and $h$. Thus, $L_{f f}=L_{f \theta}=0$. Further, we have
\begin{align}
&\left\| 
\E_{x\sim\hat{\mu}} k(\cdot,x) + \epsilon h_1 (\cdot)  +\E_{x\sim \mu_1} k(\cdot,x) - (\E_{x\sim\hat{\mu}} k(\cdot,x) + \epsilon h_2 (\cdot)  +\E_{x\sim \mu_2} k(\cdot,x))  \right\|_\rkhs \\
& \leq 
\epsilon \|h_1-h_2\|_\rkhs + \|\E_{x\sim \mu_1} k(\cdot,x) -\E_{x\sim \mu_2} k(\cdot,x)\|_\rkhs \leq \epsilon \|h_1-h_2\|_\rkhs  + \sqrt{C} \|\mu_1 - \mu_2\|_{TV},
\end{align}
i.e., $L_{f \mu}=\sqrt{C}, L_{f h}=\epsilon$.
Here we used that 
$$
{\| \E_{x\sim \mu_1} k(\cdot,x) - \E_{x\sim \mu_2} k(\cdot,x) \|_\rkhs} \leq {\sqrt{C}}\|\mu_1 - \mu_2\|_{TV}.
$$

The derivative $F'_\mu$ depends only on $f$ and $\theta$. Thus, $L_{\mu \mu}=L_{\mu h}=0$. Further, we have
\begin{align}
&\left\| 
-f_1(\cdot) + l(\theta_1;\cdot) - (-f_2(\cdot) + l(\theta_2;\cdot))  \right\|_{L^\infty} \\
& \leq 
\sqrt{C}\|f_1-f_2\|_\rkhs + L_0 \|\theta_1 - \theta_2\|_{\rkhs},
\end{align}
i.e., $L_{f \mu}=\sqrt{C}, L_{f \theta}=L_0$.
Here we used that 
\begin{multline}
    \| f_2(\cdot) - f_1(\cdot)  \|_{L^{\infty}}
    = \sup_x |f_2(x) - f_1(x)|
    = \sup_x \langle f_2 - f_1, \phi(x) \rangle_{\mathcal H}
    \\
    \leq
     \| f_2 - f_1\|_{\mathcal H} \cdot \sup_x  \|\phi(x)\|_{\mathcal H}
     \leq\sqrt{C} \cdot \| f_2 - f_1\|_{\mathcal H}.
\end{multline}

Finally, the derivative $F'_h$ depends only on $f$. Thus, $L_{h \theta}=L_{h \mu}=L_{h h}=0$. Further, we have
\begin{align}
&\left\| 
- \epsilon f_1(\cdot) - (- \epsilon f_2(\cdot)) \right\|_{\rkhs}   \leq 
\epsilon \|f_1-f_2\|_\rkhs.
\end{align}
Thus, $L_{h f}=\epsilon$.

As we see, or main assumptions in Theorem \ref{Th:ideal_MP_determ_app} hold for the DRO problem \eqref{eq:kdro_smooth}. 
Moreover, stochastic derivatives in \eqref{eq:DRO_F_theta_stoch} and \eqref{eq:DRO_F_f_stoch} satisfy Assumption \ref{Asm:SPP_gen_stoch_app}. Indeed, we have
\begin{align}
&\E_{X \sim \mu}\left\| 
\E_{x \sim \mu} \nabla_{\theta} l(\theta;x) - \nabla_{\theta} l(\theta;X) \right\|_{2}^2  \leq 
\E_{X \sim \mu}\left\| 
\nabla_{\theta} l(\theta;X) \right\|_{2}^2 \leq L_0^2,
\end{align}
\begin{align}
&\E_{\hat{X}\sim \hat{\mu}, X \sim \mu}\left\| 
\epsilon h (\cdot) +  k(\cdot,\hat{X}) + k(\cdot,X) - ( \E_{x\sim\hat{\mu}} k(\cdot,x) + \epsilon h (\cdot)  -\E_{x\sim \mu} k(\cdot,x))
 \right\|_{\rkhs}^2 \\
 & \leq 2\E_{X \sim \mu}\left\| k(\cdot,X) \right\|_{\rkhs}^2 + 2\E_{\hat{X} \sim \hat{\mu}}\left\| k(\cdot,\hat{X}) \right\|_{\rkhs}^2 \leq 4 C.
\end{align}
This allows us to apply also Theorem \ref{Th:ideal_MP_stoch_app} to the DRO problem \eqref{eq:kdro_smooth}. This leads to the following lemma, which
summarizes the applicability of our analysis to the DRO setting.
\begin{lemma}   
\label{thm:DRO_simple}
Assumptions \ref{Asm:SPP_gen_smooth_app}, \ref{Asm:SPP_gen_stoch_app} hold for the smoothed DRO problem~\eqref{eq:kdro_smooth}.
Consequently, the results of Theorems \ref{Th:ideal_MP_determ_app} and \ref{Th:ideal_MP_stoch_app} hold for the DRO problem~\eqref{eq:kdro_smooth}.
% \jz{if no space, move this lemma to the appendix}
\end{lemma}

\subsection{Proof of Guarantees for Distributionally Robust Optimization}
For notational conciseness, we denote the DRO risk associated with the decision $\theta$ as
\begin{align}
Z_{\epsilon, \emp}(\theta)
:=\sup_{\mu : \mathcal D(\mu, \emp) \le \epsilon} \E_\mu[l(\theta; x)].
\end{align}
Note that $Z_{\epsilon, \emp}(\theta)$ is a random variable since $\emp$ is a random sample from the generating distribution.
Let us denote the ideal DRO decision 
$$\THDRO := \operatorname*{\arg\inf}_\theta\sup_{\mu : \mathcal D(\mu, \emp) \le \epsilon} \E_\mu[l(\theta; x)],$$
which is often not realizable in practice.

\subsubsection{Proof of Proposition~\ref{thm:dro-subopt-guarantee}}
We restate the results and provide more detailed statements.
\begin{proposition}
    [DRO Guarantee for KMP decision sub-optimality]
    Suppose $\THKMP=\frac1N\sum_{k=0}^{N-1}\theta_k$ is the averaged solution produced by the KMP algorithm after $N$ steps.
    Then, $\forall \epsilon>0$,
    the DRO risk associated with the decision $\THKMP$
    is bounded by
    \begin{align*}
    %    \mathbb E\   
    %    \biggl[
        \sup_{\mmd(\mu, \emp) \le \epsilon} \E_\mu l(\THKMP; x) -
         \underbrace{\inf_{\theta\in\Theta}\sup_{\mmd(\mu, \emp) \le \epsilon} \E_\mu l(\theta; x)}_{\textrm{(Optimal DRO risk)}}
        % \biggr]
         \leq {O}\left(\frac{1}{N}\right).
    \end{align*}
    In the stochastic setting, we find
    \begin{align*}
        \mathbb E\   
        \biggl[
         \sup_{\mmd(\mu, \emp) \le \epsilon} \E_\mu l(\THKMP; x) -
          \underbrace{\inf_{\theta\in\Theta}\sup_{\mmd(\mu, \emp) \le \epsilon} \E_\mu l(\theta; x)}_{\textrm{(Optimal DRO risk)}}
         \biggr]
          \leq {O}\left(\frac{1}{\sqrt N}\right).
     \end{align*}
    \end{proposition}
\begin{proof}
    We note the relation
    \begin{align*}
        Z_{\epsilon, \emp}(\THDRO) =
        \inf_{\theta \in \Theta, f\in \rkhs}
            \sup_{\mu\in \mathcal M, h \in \rkhs: \|h\|_{\rkhs}\leq 1}
      F({\theta},{f},\mu,h)  
      \\
      \geq 
      \inf_{\theta \in \Theta, f\in \rkhs}
      F({\theta},{f},\bar{\mu}_N,\bar{h}_N)  
    \end{align*}
    and
    \begin{align*}
        Z_{\epsilon, \emp}(\THKMP) =
        \inf_{f\in \rkhs}
            \sup_{\mu\in \mathcal M, h \in \rkhs: \|h\|_{\rkhs}\leq 1}
      F({\theta}_N,{f},\mu,h)  
      \\
      \leq
      \sup_{\mu\in \mathcal M, h \in \rkhs: \|h\|_{\rkhs}\leq 1}
F(\bar{\theta}_N,\bar{f}_N,\mu,h)  
      .
    \end{align*}
    Taking the difference of the above expressions, 
    \begin{align}
        Z_{\epsilon, \emp}(\THKMP)  -  Z_{\epsilon, \emp}(\THDRO) \leq 
        % \\
        \sup_{\mu\in \mathcal M, h \in \rkhs: \|h\|_{\rkhs}\leq 1}
            F(\bar{\theta}_N,\bar{f}_N,\mu,h)  
            % \\
            -
            \inf_{\theta \in \Theta, f\in \rkhs}
      F({\theta},{f},\bar{\mu}_N,\bar{h}_N)  
      .
    \end{align}
    Applying the statement of Theorem~\ref{Th:ideal_MP_determ_app} completes the proof for the deterministic case.
    The stochastic case follows from Theorem~\ref{Th:ideal_MP_stoch_app}.
\end{proof}

\subsubsection{Proof of Corollary~\ref{thm:rob-mu0}}
\begin{corollary}
    [Robustness guarantee for population risk]
    Suppose the ambiguity level $\epsilon$ is chosen such that $\epsilon > \epsilon_n(\delta)$.
        Then,
        with probability $1-\delta$,
        the population risk of the decision $\THKMP$ output by the KMP algorithm after $N$ steps is estimated from above by
    \begin{align*}
            \E_{\mu_0}[l(\THKMP; x)]
        -
        \inf_{\theta\in\Theta}\sup_{\mmd(\mu, \emp) \le \epsilon} \E_\mu l(\theta; x)
        \leq 
        O\left(\frac1N\right)
       .
    \end{align*}
    The bound is $O(\frac1{\sqrt{N}})$ in expectation in the stochastic case.
\end{corollary}
\begin{proof}
With probability $1-\delta$, the empirical estimation error of MMD can be upper bounded by \cite{Tolstikhin2017} 
    \begin{equation}
        \label{eq:rate-MMD}
        \mmd(\emp, \mu_0) \leq \epsilon_n = \sqrt{\frac{C}{n}} + \sqrt{\frac{2C\log(1/\delta)}{n}}.
    \end{equation}
Since, $\epsilon > \epsilon_n(\delta)$, we trivially have
\begin{align}
    \E_{\mu_0}[l(\THKMP; x)]\leq
        Z_{\epsilon, \emp}(\THKMP)
        .
\end{align}
Applying Proposition~\ref{thm:dro-subopt-guarantee}, we find
the desired estimate
    \begin{align}
        \E_{\mu_0}[l(\THKMP; x)]
        \leq
         Z_{\epsilon, \emp}(\THKMP)
        % \\
        \leq 
        Z_{\epsilon, \emp}(\THDRO)
        + O\left(\frac1N\right)
       .
    \end{align}
\end{proof}

\subsubsection{Proof of Corollary~\ref{thm:rob-d-shift-sup}}
\begin{corollary}
    [Robustness guarantee under distribution shift]
    The risk under the worst-case distribution shift from the population distribution $\mu_0$, of any radius $\epsilon>0$,
    of the decision $\THKMP$ output by the KMP algorithm after $N$ steps
    is upper-bounded with large probability by
    \begin{align*}
        % \mathbb E 
        \sup_{\mmd(\mu, \mu_0) \le \epsilon}\E_\mu l(\THKMP; x)
        -
        \inf_{\theta\in\Theta}\sup_{\mmd(\mu, \emp) \le \epsilon} \E_\mu l(\theta; x)
        % \mathbb E Z_{\epsilon, \emp}(\THDRO)
        \leq 
        % \\
        % +
        O\left(\max\left\{\frac{1}{ N}, \frac{1}{\sqrt n}\right\}\right)
       .
    \end{align*}

    In the stochastic setting, we find
    \begin{align*}
        \mathbb E 
        \biggl[\sup_{\mmd(\mu, \mu_0) \le \epsilon}\E_\mu l(\THKMP; x)
        -
        \inf_{\theta\in\Theta}\sup_{\mmd(\mu, \emp) \le \epsilon} \E_\mu l(\theta; x)
        \biggr]
        % \mathbb E Z_{\epsilon, \emp}(\THDRO)
        \leq 
        % \\
        % +
        O\left(\max\left\{\frac{1}{ \sqrt {N}}, \frac{1}{\sqrt n}\right\}\right)
       .
    \end{align*}
    % \label{thm:rob-d-shift-sup}
\end{corollary}
\begin{proof}
The starting point of the proof is the decomposition
    \begin{align}
        \label{eq:decomposition-risk-d-shift}
        Z_{\epsilon, \mu_0}(\THKMP) - Z_{\epsilon, \emp}(\THDRO)
        =
        Z_{\epsilon, \mu_0}(\THKMP) - Z_{\epsilon, \emp}(\THKMP) 
        % \\
        + Z_{\epsilon, \emp}(\THKMP)- Z_{\epsilon, \emp}(\THDRO).
    \end{align}
We already have the upper bound for the last two terms due to Proposition~\ref{thm:dro-subopt-guarantee}.
    We now estimate the first two terms by standard concentration inequalities and the uniform law of large numbers; see\eg \cite{boucheron_theory_2005,wainwright2019high}.
    First, for any $\theta$ in the domain and RKHS function $f\in\rkhs$, with probability at least $1-\delta$,
    \begin{align*}
        \int f \dd \mu_0
        %  + \epsilon \|f \|_\rkhs +  \int \left(l(\theta;\cdot ) - f\right)\dd \mu 
        \leq 
        % \\
        \int f \dd \emp  
        % +  \epsilon \|f \|_\rkhs +  \int \left(l(\theta;\cdot ) - f\right)\dd \mu 
        + O\left(\sqrt{\frac{\log(1/\delta)}{ n}}\right).
    \end{align*}

    We consider the standard settting of an RKHS of functions with finite norms. Using uniform convergence, the term $O\left(\sqrt{\frac{\log(1/\delta)}{ n}}\right)$ does not depend on the specific function $f$ within the class $\mathcal{F}$, and the term is finite.
    Taking infimum on both sides,
    \begin{multline}
        \inf_f
        \biggl\{\int f \dd \mu_0 + \epsilon \|f \|_\rkhs +  \int \left(l(\theta;\cdot ) - f\right)\dd \mu 
        \biggr\}
        \leq 
        \\
        \inf_f
        \biggl\{\int f \dd \emp  +  \epsilon \|f \|_\rkhs +  \int \left(l(\theta;\cdot ) - f\right)\dd \mu\biggr\} + O\left(\sqrt{\frac{\log(1/\delta)}{ n}}\right).
    \end{multline}

    We use the shorthand notation 
    \begin{align}
        \tilde\mu_0:=\arg\sup_\mu\inf_f\int f \dd \mu_0 + \epsilon \|f \|_\rkhs +  \int \left(l(\theta;\cdot ) - f\right)\dd \mu 
        .
    \end{align}
    Then, by (strong) linear duality,
    \begin{multline}
        Z_{\epsilon, \mu_0}(\theta) = 
        % \\
        \sup_\mu
        \inf_f
        \biggl\{\int f \dd \mu_0 + \epsilon \|f \|_\rkhs +  \int \left(l(\theta;\cdot ) - f\right)\dd \mu 
        \biggr\}
        \\
        \leq 
        \inf_f
        \biggl\{\int f \dd \emp  +  \epsilon \|f \|_\rkhs +  \int \left(l(\theta;\cdot ) - f\right)\dd\tilde\mu_0\biggr\}
        % \\
        + O\left(\sqrt{\frac{\log(1/\delta)}{ n}}\right)
        \\
        \leq 
        \sup_\mu
        \inf_f
        \biggl\{\int f \dd \emp  +  \epsilon \|f \|_\rkhs +  \int \left(l(\theta;\cdot ) - f\right)\dd \mu\biggr\} 
        % \\
        + O\left(\sqrt{\frac{\log(1/\delta)}{ n}}\right)
        \\
        = Z_{\epsilon, \emp}(\theta) 
        + O\left(\sqrt{\frac{\log(1/\delta)}{ n}}\right)
    \end{multline}
    holds with probability at least $1-\delta$.

    In summary, plugging in $\theta=\THKMP$, 
    \begin{align*}
        \mathbb P \Biggl(Z_{\epsilon, \mu_0}(\THKMP) 
        -
        Z_{\epsilon, \emp}(\THKMP) 
        \geq  O\left(\sqrt{\frac{\log(1/\delta)}{ n}}\right)\Biggr) \leq \delta.
    \end{align*}
    
    Furthermore, to derive the expected bound, we use the Gaussian integral to integrate the tail bound, combining with \eqref{eq:decomposition-risk-d-shift} and Proposition~\ref{thm:dro-subopt-guarantee}. We find the desired statement in expectation
    \begin{align}
        \E\biggl[
            Z_{\epsilon, \mu_0}(\THKMP) - Z_{\epsilon, \emp}(\THDRO)
        \biggr]
        \leq  O\left(\max\left\{\frac{1}{ N}, \frac{1}{\sqrt n}\right\}\right).
    \end{align}
    The stochastic case follows from the previous settings.
\end{proof}

\newpage